\documentclass{amsart}
\usepackage{amssymb} 
\usepackage{amsbsy}
\usepackage{amscd}
\usepackage{amsmath}
\usepackage{amsthm}
\usepackage{amsxtra}
\usepackage{latexsym}
\usepackage[mathscr]{eucal}
\usepackage{upgreek}
\usepackage{wasysym}
\usepackage[all,cmtip]{xy}

\usepackage{xcolor}
\definecolor{rouge}{rgb}{0.85,0.1,.4}
\definecolor{bleu}{rgb}{0.1,0.2,0.9}
\definecolor{violet}{rgb}{0.7,0,0.8}
\usepackage[dvipdfmx,colorlinks=true,linkcolor=bleu,urlcolor=violet,citecolor=rouge]{hyperref}
\newcommand{\Miura}{\Upupsilon}
\newcommand{\Wak}[2]{\mathbb{W}_{#2}(#1)}
\renewcommand{\L}{\mathbb{L}}
\newcommand{\V}{\mathbb{V}}
\newcommand{\M}{\mathbb{M}}
\newcommand{\wh}{\widehat}

\newcommand{\bra}{{\langle}}
\newcommand{\ket}{{\rangle}}

\newcommand{\Lam}{\Lambda}
\newcommand{\germ}{\mathfrak}
\newcommand{\cprime}{$'$}
\newcommand{\on}{\operatorname}
\newcommand{\+}{\mathop{\oplus}}
\renewcommand{\*}{{\otimes}}
\newcommand{\mc}{\mathcal}
\newcommand{\mf}{\mathfrak}

\newcommand{\g}{\mf{g}}
\newcommand{\h}{\mf{h}}

\newcommand{\affg}{\widehat{\mf{g}}}

\newcommand{\affh}{\widehat{\mf{h}}}
\newcommand{\isomap}{{\;\stackrel{_\sim}{\to}\;}}
\newcommand{\Z}{\mathbb{Z}}
\newcommand{\C}{\mathbb{C}}
\newcommand{\N}{\mathbb{N}}
\newcommand{\Q}{\mathbb{Q}}
\newcommand{\W}{\mathscr{W}}
\newcommand{\ra}{\longrightarrow}
\newcommand{\lam}{\lambda}

\def\leq{\leqslant}
\def\geq{\geqslant}

\DeclareMathOperator{\gr}{gr}

\DeclareMathOperator{\Hom}{Hom}

\theoremstyle{theorem}
\newtheorem{Th}{Theorem}[section]
\newtheorem{MainTh}{Main Theorem}
\newtheorem{Pro}[Th]{Proposition}

\newtheorem{Lem}[Th]{Lemma}

\newtheorem{Co}[Th]{Corollary}

\theoremstyle{remark}
\newtheorem{Def}[Th]{Definition}
\newtheorem{Rem}[Th]{Remark}

\title{W-algebras as coset vertex algebras}

\author{Tomoyuki Arakawa}
\address{Research Institute for Mathematical Sciences, Kyoto University,
Kyoto 606-8502 JAPAN}
\email{arakawa@kurims.kyoto-u.ac.jp}
\author{Thomas Creutzig}
\address{Department of Mathematical and Statistical Sciences, University of Alberta, Edmonton, AB T6G 2G1 Canada}
\email{creutzig@ualberta.ca}
\author{Andrew R. Linshaw}
\address{Department of Mathematics, University of Denver
Denver, CO 80208}
\email{andrew.linshaw@du.edu}
\thanks{This work was partially supported by JSPS KAKENHI Grants
(\#17H01086 and \#17K18724 to T. Arakawa), an NSERC Discovery Grant (\#RES0019997 to T. Creutzig), and a grant from the Simons Foundation (\#318755 to A. Linshaw)}

\begin{document}
\begin{abstract}
We prove the long-standing conjecture on the coset construction of the minimal series principal $W$-algebras 
of $ADE$ types in full generality.
We do this by first
establishing Feigin's conjecture
on  the coset realization of the universal principal $W$-algebras,
which are not necessarily simple. As consequences, the unitarity of the \lq\lq discrete series" of principal $W$-algebras is established, a second coset realization of rational and unitary $W$-algebras of type $A$ and $D$ are given and the rationality of Kazama-Suzuki coset vertex superalgebras is derived. 
\end{abstract}
\maketitle

\section{Introduction}
Let $\g$ be a simple Lie algebra.
For each nilpotent element $f\in \g$ and $k\in \C$,
one associates the $W$-algebra $\W^k(\g,f)$ at level $k$
via quantum Drinfeld-Sokolov reduction \cite{FF90,KacRoaWak03}.
In the instance that $f$ is  a principal nilpotent element  $\W^k(\g,f)$ is called the universal principal $W$-algebra of $\g$ at level $k$
and denoted by  $\W^k(\g)$. These $W$-algebras have appeared prominently in various problems of mathematics and physics as 
the conformal field theory to higher spin gravity duality \cite{GG}, the AGT correspondence \cite{AGT,SchVas13,BraFinNak16}, the (quantum) geometric Langlands program \cite{Fre07,Gai16,AgaFreOko, CG, Gai18,  FG} and integrable systems \cite{B89, D03, De-KacVal13, BakMil13}.

Let $\W_k(\g)$ be the unique simple graded quotient of $\W^k(\g)$.
It has been conjectured in \cite{FKW92}
and was proved by the first named author \cite{Ara09b,A2012Dec} 
that $\W_k(\g)$ is rational and lisse for some special values of $k$.
These $W$-algebras are called the {\em minimal series principal $W$-algebras}
since in the case that $\g=\mf{sl}_2$ they are exactly the minimal series Virasoro vertex algebras \cite{BPZ84,BeiFeiMaz,Wan93}.
As in the case of the Virasoro algebra,
a minimal series principal $W$-algebra is not necessarily unitary.
However,
 in the case that $\g$ is simply laced,
 there exists a sub-series called the {\em discrete series}
which are conjectured to be unitary.

It has been believed in physics since 1988 that the discrete series principal $W$-algebras
can be realized by the coset construction \cite{GodKenOli86}
from integrable representations of the affine Kac-Moody algebra $\affg$
 \cite{Bais, FatLyk88}. 
 Note that 
 the validity of this belief immediately proves the unitarity of the discrete series of $W$-algebras.
 The conjectural   character formula of Frenkel, Kac and Wakimoto \cite{FKW92} 
 of minimal series representations of $W$-algebras  that was proved in \cite{Ara07}
 together with the character formula of Kac-Wakimoto \cite{KacWak90}
 of branching rules 
 proves the matching of characters,
which gives strong evidence to this belief.
In fact these character formulas  give
 an even stronger conjecture that 
 all minimal series principal $W$-algebras of $ADE$ types should be realized by the coset construction
 if we consider more general representations of $\affg$,
 namely,
admissible representations \cite{KacWak89}.

One of
the aims of the present  paper is to prove 
this 
conjectural realization of the minimal series principal $W$-algebras 
in full generality.

\subsection{Main Theorems}
Let us formulate our result more precisely.
Let $V_k(\g)$ be the universal affine vertex algebra associated to $\g$ at level $k$,
and
denote by $L_k(\g)$  the unique simple graded quotient of $V_k(\g)$.
Suppose that  $k$ is an admissible level for $\affg$,
that is,
$L_k(\g)$ is an admissible representation.
Consider the tensor product vertex algebra
$L_k(\g)\* L_1(\g)$.
The invariant subspace
$(L_k(\g)\* L_1(\g))^{\g[t]}$ 
with respect to the 
diagonal action of $\g[t]$
is naturally a vertex subalgebra of $L_k(\g)\* L_1(\g)$,
consisting of elements whose Fourier modes 
commute with the diagonal action of $\affg$.
This is an example of 
 {\em coset vertex algebras}.

\begin{MainTh}\label{Main1}
Let $\g$ be simply laced,
and let  $k$ be an admissible level for $\affg$.
Define the rational number $\ell$ by
the formula 
\begin{align}
\ell+h^{\vee}=\frac{k+h^{\vee}}{k+h^{\vee}+1},
\label{eq:ellintro}
\end{align}
which is a non-degenerate admissible level for $\affg$
so that $\W_{\ell}(\g)$ is a minimal series $W$-algebra.
We have the vertex algebra isomorphism
$$\W_{\ell}(\g)\cong  (L_k(\g)\* L_1(\g))^{\g[t]},$$
and 
$L_{k+1}(\g)$ and 
 $\W_{\ell}(\g)$ form a dual pair  in $ L_k(\g)\* L_1(\g)$.
 \end{MainTh}

For $\g=\mf{sl}_2$ and $k$ a non-negative integer,
Theorem \ref{Main1}  recovers a celebrated result of 
Goddard, Kent and Olive 
\cite{GodKenOli86},
which is  known as 
the {coset construction} (or the {\em{GKO construction}}) of the discrete unitary series of the Virasoro algebra.
Theorem \ref{Main1} was extended to the case of an arbitrary admissible level  $k$
for $\g=\mf{sl}_2$
by Kac and Wakimoto
 \cite{KacWak90}.
 For a higher rank $\g$,
Theorem \ref{Main1} has been proved only in some  special cases:
 $\g=\mf{sl}_n$ and $k=1$ by Arakawa, Lam and Yamada \cite{ALY17};
for $\g=\mf{sl}_3$ and $k\in \Z_{\geq 0}$ by Arakawa and Jiang \cite{AraJia}.

Theorem \ref{Main1}
realizes an  arbitrary minimal series $W$-algebra 
of $ADE$ types as 
the coset  $(L_k(\g)\* L_1(\g))^{\g[t]}$ for some admissible level $k$.
The discrete series $W$-algebras corresponds to the cases that 
 $k$ is a non-negative integer.
 As we have already mentioned above, 
with Main Theorem \ref{Main1} we are able to prove the unitarity of the discrete series of $W$-algebras,
 see Theorem \ref{Th:Unitarity}.

We also note that
Theorem \ref{Main1}
 is the key starting assumption of the conformal field theory to higher spin gravity correspondence of \cite{GG}.

 \smallskip
 
 Since
Kac and Wakimoto \cite{KacWak90}
 have already confirmed the 
 matching of characters,
the essential step
 in proving   Main Theorem \ref{Main1}
 is to define the action of $\W^{\ell}(\g)$ on 
 $(L_k(\g)\* L_1(\g))^{\g[t]}$,
 which is highly non-trivial since
 there is no closed presentation of   $\W^{\ell}(\g)$
 by generators and relations (OPEs) for a general $\g$.
 We overcome this difficulty by establishing the following assertion
that has been conjectured by B. Feigin  (cf.\ \cite{FeiJimMiw16}).
 \begin{MainTh}[Theorem \ref{Th:universalGKO}]\label{Main2}
Let $\g$ be simply laced,
 $k+h^{\vee}\not\in \Q_{\leq 0}$,
 and define $\ell\in \C$ by the formula \eqref{eq:ellintro}.
We have
the vertex algebra isomorphism
 $$\W^{\ell}(\g)\cong (V_k(\g)\* L_1(\g))^{\g[t]}.$$
 Moreover,
 $\W^{\ell}(\g)$ and $V_{k+1}(\g)$ form a dual pair in $V_k(\g)\* L_1(\g)$ if $k$ is generic.
\end{MainTh}
The advantage of replacing $\W_{\ell}(\g)$ by the universal $W$-algebra $\W^{\ell}(\g)$ lies in the fact that 
one can use the description of $\W^{\ell}(\g)$
in terms of screening operators,
at least for a generic $\ell$.
Using such a description,
we are able to establish the statement of Main Theorem \ref{Main2}
for deformable families \cite{CL} of $\W^{\ell}(\g)$ and $ (V_k(\g)\* L_1(\g))^{\g[t]}$,
see Section \ref{section:coset-vs-W}
for the details.
The main tool here is a property of the semi-regular bimodule 
obtained in \cite{A-BGG}, see Proposition \ref{Pro:key-iso}.

\smallskip

The second part of our main result is the branching rules, i.e. the decomposition of modules of $V_k(\g)\* L_1(\g)$ and $L_k(\g)\* L_1(\g)$ into modules of the tensor product of the two commuting subalgebras. For this we need to introduce some notation that is also explained in full detail in the main body of the work. 

Let $P_+$ be the set of dominant weights of $\g$, $Q$ its root lattice and $\rho^\vee$ half the sum of positive coroots. For $\lam\in P_+$
define
$\V_k(\lam):=U(\affg)\otimes_{U(\g[t]\+ \C K)} E_{\lam}$,
where $E_{\lam}$ is the irreducible finite-dimensional $\g$-module
with highest weight $\lam$ regarded as a $\g[t]$-module on which $\g[t]t$ acts trivially and $K$ acts by multiplication with the level $k\in\mathbb C$. Let $\L_k(\lam)$ be the simple quotient of $\V_k(\lam)$ and for $m\in \N$,
let $P^m_+$ be the set of highest-weights such that $\{\L_m(\lam)\mid \lam \in P^m_+\}$
gives  the complete set of isomorphism classes of irreducible integrable
representation of  $\affg$ of level $m$.
We denote by $\chi_{\lam}$ the central character associated to the weight $\lam$, see \eqref{eq:centralchar} for details. 
Let $\mathbf{M}_k(\chi_{\lam})$ be the Verma module of $\W^k(\g)$ with highest weight $\chi_{\lam}$ (see Section
\ref{section:Miura}) and denote by $\mathbf{L}_k(\chi_{\lam})$ be the unique irreducible (graded) quotient of $\mathbf{M}_k(\chi_{\lam})$.
\begin{MainTh} \label{main3} Define $\ell\in \C$ by \eqref{eq:ellintro}, then
the following branching rules hold:
\begin{enumerate}
\item 
Let $\mu\in P_+^{p-h^{\vee}}$,
$\nu\in P_+^1$,
We have
\begin{align*}
\L_k(\mu)\otimes \L_1(\nu)\cong \bigoplus_{\lam\in P^{p+q-h^{\vee}}_+\atop \lam-\mu-\nu\in Q}
\L_{k+1}(\lam)\otimes  \mathbf{L}_{\ell}(\chi_{\mu-(\ell+h^{\vee})\lam})
\end{align*}
as $L_{k+1}(\g)\* \W_{\ell}(\g)$-modules.
\item
Suppose that $k\not\in \mathbb{Q}$.
For $\lam,\mu\in P_+$ and $\nu\in P_+^1$,
we have
\begin{align*}
\V_k(\mu)\* \L_1(\nu)=\bigoplus_{\substack{\lam\in P_+\\ \lam-\mu-\nu\in Q}}\V_{k+1}(\lam)\*  \mathbf{L}_\ell(\chi_{\mu-(\ell+h^{\vee})\lam})
\end{align*}
as  $V_{k+1}\*\W^{\ell}(\g)$-modules.
\end{enumerate}
\end{MainTh}
The generic decomposition is Theorem \ref{thm:genericdecomp} and the one at admissible level is Theorem \ref{Th:minimal-series-modules}.

We note that the $\W^\kappa(\g)$-modules 
$\mathbf{L}_{\ell}(\chi_{\mu-(\ell+h^{\vee})\lam})
$ that appear in Theorem \ref{thm:genericdecomp} 
play a crucial role in the quantum geometric Langlands program and gauge theory \cite{CG, Gai18,FG}.  

\subsection{Corollaries}

Since the minimal series $W$-algebras are rational
and lisse \cite{Ara09b,A2012Dec},
Theorem \ref{Main1} 
establishes  the rationality of a large class of coset vertex algebras.
We are able to derive further rationality statements as Corollaries from Main Theorem \ref{Main1}.
It is worth mentioning that the rationality problem is wide open for a general coset vertex algebra.

The first one is Corollary \ref{Co:rationalprod}, saying that
\begin{Co} 
Let $\g$ be simply laced,
$k$ an admissible number,
$n$ a positive integer.
Then
the coset vertex algebra
$(L_k(\g)\* L_1(\g)^{\otimes n})^{\g[t]}$ 
 is rational and lisse.  Here $L_1(\g)^{\otimes n}$ denotes the tensor product of $n$ copies of $L_1(\g)$, and $\g[t]$ acts on $L_k(\g)\* L_1(\g)^{\otimes n}$ diagonally.
 In particular, $(L_m(\g)\* L_1(\g)^{\otimes n})^{\g[t]}$  is rational and lisse for any positive integers $m,n$.
 \end{Co}
Second, we establish level-rank dualities of types $A$ and $D$.
 For this let $L_k(\mf{gl}_n)=L_k(\mf{sl}_n)\* \mc{H}$ be the simple affine vertex algebra associated to $\mf{gl}_n$ at level $k$,
where $\mc{H}$ is the rank $1$ Heisenberg vertex algebra.
The natural embedding $\mf{gl}_{n}\hookrightarrow \mf{gl}_{n+1}$
gives rise to the vertex algebra embedding
$L_k(\mf{gl}_n)\hookrightarrow L_k(\mf{gl}_{n+1})$.
The invariant subspace
$ L_k(\mf{gl}_{n+1})^{\mf{gl}_n[t]}$
is  the coset vertex subalgebra of $ L_k(\mf{gl}_{n+1})$,
consisting of elements whose Fourier modes 
commute with the action of $L_k(\mf{gl}_n)$. Theorem \ref{th:levelrankA} implies:

\begin{Co}\label{Th:affine-Gelfand-Tsetlin}{ \rm(Level-rank duality of type $A$)}
For positive integers  $k,n$ one has
\begin{align*}
L_k(\mf{gl}_{n+1})^{\mf{gl}_{n}[t]}\cong  \W_{\ell}(\mf{gl}_k),
\end{align*}
 where $\ell$ is the non-degenerate admissible number defined by the formula  $$\ell+k= \frac{k+n}{k+n+1}.$$
In particular,
$L_k(\mf{sl}_{n+1})^{\mf{gl}_{n}[t]}\cong  \W_{\ell}(\mf{sl}_k)$,
and is simple, rational and lisse.
\end{Co}
In other words,
the simple affine vertex algebra 
$L_k(\mf{gl}_{n})$ contains a simple vertex subalgebra isomorphic to
\begin{align*}
\W_{\ell_1}(\mf{gl}_k)\otimes \W_{\ell_2}(\mf{gl}_k)\otimes \dots \otimes \W_{\ell_n}(\mf{gl}_k)
\end{align*}
with $\ell_i+k=(k+n-i)/(k+n-i+1)$,
which may be regarded as an affine, non-commutative analogue of the Gelfand-Tsetlin subalgebra
of $U(\mf{gl}_n)$.
Note that iterating this coset construction tells us that $L_k(\mf{sl}_{n})^{\mf{gl}_{m}[t]}$ for positive integers $m<n$ is simple, rational and lisse (Corollary \ref{Co:rationalA}).
Note also that for $n=1$,
$L_k(\mf{sl}_{n+1})^{\mf{gl}_{n}[t]}$ is the $\mf{sl}_2$-parafermion vertex algebra
and Theorem \ref{Th:affine-Gelfand-Tsetlin} has been proved in 
 \cite{ALY17}.

There is a similar statement for type $D$. It is Theorem \ref{th:levelrankD} and it implies:

\begin{Co}{ \rm(Level-rank duality of type $D$)}
Let  $k,n$ be positive integers and $k$ even. Let $G = \mathbb Z/2\mathbb Z$ for $n$ odd and $G = \mathbb Z/2\mathbb Z \times  \mathbb Z/2\mathbb Z$ for $n$ even and let $\omega_1$ be the first fundamental weight of $\mf{so}_{n+1}$. Then 
\begin{align*}
\left(\left(L_k(\mf{so}_{n+1})\oplus \L_n(n\omega_1)\right)^{\mf{so}_{n}[t]}\right)^G \cong  \W_{\ell}(\mf{so}_k),
\end{align*}
 where $\ell$ is the non-degenerate admissible number defined by the formula  $$\ell+k-2= \frac{k+n-2}{k+n-1}.$$
 In particular,
$L_k(\mf{so}_{n+1})^{\mf{so}_{n}[t]}$ is simple, rational and lisse.
\end{Co}
Iterating this coset construction tells us that $L_k(\mf{so}_{n})^{\mf{so}_{m}[t]}$ for positive integers $m, n, k$ such that $2 \leq m<n$ and $k$ even is simple, rational and lisse (Corollary \ref{Co:rationalD}).

Main Theorem \ref{Main1}
has as another Corollary rationality of certain coset vertex superalgebras.
The type $A$ case is called Kazama-Suzuki coset in physics and it is our 
Corollary \ref{Co:supercoset}. We also have a type $D$ case which is Corollary \ref{Co:supercosetD}.
These are important since the corresponding superconformal field theories for the Kazama-Suzuki cosets can be used as building blocks for sigma models in string theory \`a la Gepner \cite{Gep}, and also its rationality is the starting assumption of the superconformal field theory to higher spin supergravity correspondences of  \cite{CHR, CHR2}.

\subsection{Gauge Theory and the quantum geometric Langlands program}

Recently there has been considerable interest in connecting four-dimensional supersymmetric gauge theories, the quantum geometric Langlands program and vertex algebras. 
On the vertex algebra side, one is interested in certain {\em master chiral algebras} that serve as a kernel for the quantum geometric Langlands correspondence \cite{Gai18} and at the same time as a corner vertex algebra for the junction of topological Dirichlet boundary conditions in gauge theory \cite{CG, FG}. Roughly speaking, physics predicts the existence of vertex algebra extensions of tensor products of vertex algebras associated to $\g$, the Langlands dual ${}^L\g$ of $\g$, and the coupling $\Psi$. Different such extensions are expected to be related via coset constructions and quantum Drinfeld-Sokolov reductions. These vertex algebra extensions are then expected to imply equivalences of involved vertex tensor categories and also spaces of conformal blocks (twisted $D$-modules). We refer to \cite{FG, Gai18} for recent progress in this direction.

We will now explain that our Main Theorem \ref{main3} (b) proves two physics conjectures. 
The gauge theory is specified by a coupling $\Psi$, a generic complex number, and a compact Lie group $G$, the gauge group. Let $\g$ be the Lie algebra of $G$ and assume $\g$ is simply-laced. 
Let $n$ be a positive integer and define $k=\Psi-h^\vee$, then the conjectural junction vertex algebra for the Dirichlet boundary conditions $B_{n, 1}^D$ and $B^D_{0, 1}$ is
\begin{align*}
A^{(n)}[G, \Psi] \cong \bigoplus_{\lam\in P_+\atop \lam\in Q}
\V_{k}(\lam)\otimes \V_{\ell}(\lam), \qquad \qquad \frac{1}{k+h^\vee} +\frac{1}{\ell+h^\vee}=n;
\end{align*}
as a module for $V_k(\g) \otimes V_\ell(\g)$. The existence of this simple vertex algebra is presently only established for $\g=\mathfrak{sl}_2$ and $n=1, 2$ \cite{CG, CGL}.
Our main Theorem \ref{main3} fits very nicely into this context and confirms two physics predictions of \cite[Section 2 and 3]{CG}: First, one takes the case $n=1$ and then notices that $V_{k-1}(\g)\otimes L_1(\g)$ is isomorphic to the $V_{k}(\g)\otimes \W^\ell(\g)$-module obtained by replacing the Weyl modules $\V_\ell(\lam)$ of $V_\ell(\g)$ in $A^{(1)}[G, \Psi]$ by the corresponding modules of $\W^{\ell}(\g)$ obtained via quantum Drinfeld-Sokolov reduction.  Second, the junction vertex algebra between Neumann and Dirichlet boundary conditions is the affine vertex algebra $V_{\Psi-h^\vee}(\g)$ and the one between Neumann and Neumann boundary conditions is the regular $W$-algebra of $\g$. Concatenating two such junction vertex algebras gives an extension of these two vertex algebras that is precisely of the form of Main Theorem \ref{main3} (b) and so isomorphic to $V_{\Psi-1-h^\vee}(\g)\otimes L_1(\g)$, confirming another physics prediction. 

We note that the relations between vertex algebras, physics, and quantum geometric Langlands is rich and we will continue to prove further vertex algebra statements in this context as e.g. \cite{CGL}. The methods of our work should be quite helpful for that.

\subsection*{Notation}
For a vertex algebra $V$ and a vertex subalgebra $W\subset V$,
the {\em commutant} of $W$ in $V$,
or the coset of $V$ by $W$, is the vertex subalgebra of $V$ defined by
\begin{align*}
\on{Com}(W,V)
=\{v \in V\mid [w_{(m)},v_{(n)}]=0\ \text{for all }m,n\in \Z,\ w\in W\}\\
=\{v\in V\mid w_{(n)}v=0\ \text{for all }n\in \Z_{\geq 0}, \ w\in W\}.
\end{align*}
Vertex subalgebras $W_1$ and $W_2$ of a vertex algebra $V$ are said to form a {\em dual pair} 
if they are mutually commutant, that is,
\begin{align*}
W_2=\on{Com}(W_1,V)\quad \text{and}\quad W_1=\on{Com}(W_2,V).
\end{align*}

\subsection*{Acknowledgement}
This work  started when we visited Perimeter Institute for Theoretical Physics, Canada,
for the conference
 ^^ ^^ Exact operator algebras in superconformal field theories" in  December 2016.
 We thank the organizers of the conference and the institute.
 The  first named author would like to thank
MIT for its
 hospitality 
 during his visit  from February 2016 to January 2018.

\section{The universal commutant vertex algebra}\label{section:universal-coset}
Let $\g$ be a simple Lie algebra,
$\g=\mf{n}_-\+\mf{h}\+\mf{n}$ a triangular decomposition,
$\Delta$ the set of roots of $\g$,
$\Delta_+$ the set of positive roots of $\g$,
$W$ the Weyl group of $\g$,
$Q$ the root lattice of $\g$,
$Q^{\vee}$ the coroot lattice of $\g$,
$P$ the weight lattice of $\g$,
$P^{\vee}$ the coweight lattice of $\g$,
$P_+\subset P$ the set of dominant weights of $\g$,
$P_+^{\vee}\subset P^{\vee}$ the set of dominant coweights of $\g$.
For $\lam\in P_+$,
denote by $E_{\lam}$ the irreducible finite-dimensional representation of $\g$ with highest weight $\lam$.

Let $\affg=\g\otimes \C[t,t^{-1}]\+ \C K$, the affine Kac-Moody algebra associated to $\g$,
$\affh=\h\+ \C K$ the Cartan subalgebra of $\affg$,
$\affh^*=\h^*\+ \C \Lam_0$ the dual of $\affh$.
Set
$\widehat{\mf{n}}=\mf{n}+ \g[t]t$,
$\widehat{\mf{n}}_-=\mf{n}_-+ \g[t^{-1}]t$,
$\widehat{\mf{b}}=\affh\+ \widehat{\mf{n}}$,
so that $\affg=\widehat{\mf{n}}_-\+ \affh\+\widehat{\mf{n}}=\widehat{\mf{n}}_-\+\widehat{\mf{b}}$.

Let $\widetilde{\mf{h}}=\affh\+ \C D$ be the extended Cartan subalgebra (\cite{Kac90})
of $\affg$.
Then 
$\widetilde{\mf{h}}^*=\h^* \+ \C \Lam_0\+ \C \delta$,
where 
$\Lam_0(K)=\delta(D)=1$,
$\Lam_0(\h\+ \C D)=\delta(\h\+ K)=0$.
Let $\widehat{\Delta}\subset \widetilde{\mf{h}}^*$ be the set of roots of $\affg$,
$\widehat{\Delta}^{re}$, the set of real roots, 
$\widehat{\Delta}_+$, the set of positive roots,
$\widehat{\Delta}^{re}_+=\widehat{\Delta}_+\cap \widehat{\Delta}^{re}$.
Let $\widehat{W}=W\ltimes Q^{\vee}$, the affine Weyl group of $\affg$.
For $\mu\in Q^{\vee}$,
the corresponding element in $\widehat{W}$ is denote by $t_{\mu}$.
The dot action of $\widehat{W}$ acts on ${\affh}^*$ 
is given by
$w\circ \Lam=w(\Lam+\hat{\rho})-\hat{\rho}$,
where $\hat{\rho}=\rho+h^{\vee}\Lam_0$
and $\rho$ is the half sum of positive roots of $\g$.

\smallskip

Let $T$ be an integral $\C[K]$-domain
with the structure map $\tau:\C[K]\ra T$.
Define
\begin{align*}
V_T(\g)=U(\affg)\otimes_{U(\g[t]\+ \C K)}T,
\end{align*}
where
$T$  is regarded as a $\g[t]$-module on which
$\g[t]$ acts trivially
and $K$ acts as the multiplication by $\tau(K)$.
There is a unique vertex algebra structure on
$V_T(\g)$ such that 
$\mathbf{1}=1\*1$ is the vacuum vector,
$Y(T\mathbf{1},z)=1 \*T$,
$Y(xt^{-1}\mathbf{1},z)=x(z):=\sum_{n\in \Z}((xt^n)\* 1)z^{-n-1}$,
$x\in \g$.
$V_T(\g)$ is called 
 the {\em universal affine vertex algebra associated to $\g$ over $T$},
 cf.\ \cite{De-Kac06,CL}.
 
 If $T=\C$ and $\tau(K)=k\in \C$,
$V_T(\g)$ is the  {\em universal affine vertex algebra
associated to $\g$ at level $k$} and,
is  denoted also by $V_k(\g)$.
 The simple graded quotient of $V_k(\g)$ is denoted by $L_k(\g)$.
 
 A $V_T(\g)$-module is the same as a smooth $U(\affg)\* T$-module on which
 $K$ acts as the multiplication by $\tau(K)$.
 Here  a  $U(\affg)\* T$-module $M$ is called smooth if
 $x(z)$, $x\in \g$, is a field on $M$,
 that is,
  $x(z)m\in M((z))$ for all $m\in M$.

In the rest of this section we assume that 
 $\tau(K)+h^{\vee}$ is invertible in $T$,
 so that
the vertex algebra $V_T(\g)$ is conformal by the Sugawara construction.
The field corresponding to the conformal vector $\omega\in V_T(\g)$ is given by
\begin{align*}
L(z)=\sum_{n\in \Z}L_nz^{-n-2}:=\frac{1}{2(\tau(K)+h^{\vee})}\sum_i: x_i(z)x^i(z):.
\end{align*}
The central charge  of $V_T(\g)$ 
equals to  $\tau(K)\dim \g/(\tau(K)+h^{\vee})\in T$.

Set $\h_T=\h\* T$,
$\h_T^*=\on{Hom}_T(\h_T,T)$.
We regard $\h^*$ as a subset of $\h_T^*$
by the 
natural embedding
$\h^*\hookrightarrow \h_T^*$,
$\lam\mapsto( h\* a\mapsto a\lam(h))$.
For a $V_T(\g)$-module $M$,
$\lam\in \h^*_T$,
and $\Delta\in T$,
put
\begin{align*}
&M^{\lam}=\{m\in M\mid hm=\lam(h)m,\  \forall h\in{\h}\},
\\
&M_{[\Delta]}=\{m\in M\mid (L_0-\Delta)^r m=0,\  r\gg 0\},\quad 
M_{\Delta}=\{m\in M\mid L_0m =\Delta m\},\\
&M_{[\Delta]}^\lam=M^{\lam}\cap M_{[\Delta]}.
\end{align*}
$M$ is called as a {\em weight module} if $M=\bigoplus\limits_{\lam\in {\h}^*_T,\
\Delta\in T}M^{\lam}_{[\Delta]}$.

The {\em Kazhdan-Lusztig category} $\on{KL}_T$ over $T$
is the category of all $V_T(\g)$-modules $M$
such that (1) $M$ is a weight module
and $M^{\lam}=0$ unless $\lam\in P$,
(2)  $U(\g[t])\*T. m$ is finitely generated as a $T$-module for all $m\in M$,
(3) $M$ is direct sum of finite-dimensional $\g$-modules. 

For $\lam\in P_+$,
define
\begin{align*}
\V_T(\lam):=U(\affg)\otimes_{U(\g[t]\+ \C K)} (E_{\lam}\* T)\in \on{KL}_T,
\end{align*}
where $E_{\lam}$ is the irreducible finite-dimensional $\g$-module
with highest weight $\lam$ regarded as a $\g[t]$-module on which $\g[t]t$ acts trivially.
We have
$$\V_T(\lam)=\bigoplus_{d\in \Z_{\geq 0}}\V_T(\lam)_{h_{\lam}+d},\quad 
\V_T(\lam)_{h_{\lam}}=E_{\lam}\* T,$$
and each $\V_T(\lam)_{h_{\lam}+d}$ is a free $T$-module of finite rank,
where
\begin{align}
h_{\lam}:=\frac{(\lam+2\rho|\lam)}{2(\tau(K)+h^{\vee})}=\frac{1}{2(\tau(K)+h^{\vee})}(|\lam|^2+\sum_{\alpha\in \Delta_+}(\lam|\alpha))\in
T.
\label{eq:lowest-weight}
\end{align}
Note that 
$V_T(\g)\cong \V_T(0)$ as an object of $\on{KL}_T$.

If $T=\C$ and $\tau(K)=k\in \C$,
we write $\on{KL}_k$ for $\on{KL}_T$
and $\V_k(\lam)$ for $\V_T(\lam)$.
\begin{Lem}\label{lem:projective}
Suppose that 
$k+h^{\vee}\in \C\backslash \mathbb{Q}_{\leq 0}$.
Then
$
\on{Hom}_{\affg}(V_k(\g), M)\cong M_{[0]}
$
for any $M\in \on{KL}_k$.
In particular $V_k(\g)$ is projective in $\on{KL}_k$.

\end{Lem}
\begin{proof}
Although this statement is well-known,
we include a proof for completeness. 
The assumption and the second formula of  \eqref{eq:lowest-weight}
imply that 
 $h_{\lam}\not \in \mathbb{R}_{<0}$
  for 
$\lam\in P_+$,
 and that 
$h_{\lam}=0$ if and only $\lam=0$.
It follows that 
$M_{[-n]}=0$ for all $n\in \Z_{>0}$,
and thus,
$L_0$ acts as 
$\Omega/2(k+h^{\vee})$ on $M_{[0]}$,
where $\Omega$ is the Casimir element of $U(\g)$.
As $\Omega$ acts semisimply on an object of $\on{KL}_k$,
it follows that $M_{[0]}=M_0$.
Moreover,
$M_{[0]}$ is a direct sum of trivial representations of $\g$.
Therefore,
we conclude that  $M_{[0]}=M_0^{\g[t]}:=M_0\cap M^{\g[t]}$.
Hence 
$\on{Hom}_{\affg}(V_k(\g), M)\cong M_0^{\g[t]}=M_{[0]}$,
by the Frobenius reciprocity.
The last statement 
follows from the fact that 
$M\mapsto M_{[0]}$ is an exact functor
from $\on{KL}_k$ to the category of $\g$-modules.
\end{proof}

We now assume that 
$\tau(K)+h^{\vee}+1$ is invertible in $T$
as well as $\tau(K)+h^{\vee}$.
For $a\in \C$, let  $T+a$ denote the $\C[K]$-algebra
$T$ with the structure map $K\mapsto \tau(K)+a$.

Consider the tensor product
$V_T(\g)\* L_1(\g)$.
As $V_T(\g)$
is free over $U(\g[t^{-1}]t^{-1})$,
so is $V_T(\g)\* L_1(\g)$.
It follows that we have  the vertex algebra embedding
\begin{align*}
V_{T+1}(\g)\hookrightarrow V_T(\g)\* L_1(\g),\quad
u\mathbf{1}\mapsto \Delta(u)(\mathbf{1}\*\mathbf{1})\quad (u\in U(\affg)),
\end{align*}
where $\Delta(u)$ denotes the coproduct of $U(\affg)$.

The following assertion is clear.
\begin{Lem}
$V_T(\g)\* L_1(\g)$ 
is an object of $\on{KL}_{T+1}$ 
as a $V_{T+1}(\g)$-module.
Moreover, each weight space
$(V_T(\g)\* L_1(\g))^{\lam}_{\Delta}$ is a free $T$-module.
\end{Lem}

Define
the vertex algebra $\mc{C}_T(\g)$ by
\begin{align}
\mc{C}_T(\g):=\on{Com}(V_{T+1}(\g), V_T(\g)\* L_1(\g)).
\label{eq:universal-coset}
\end{align}
Note that
$$\mc{C}_T(\g)\cong  \Hom_{\on{KL}_{T+1}}(V_{T+1}(\g),V_T(\g)\* L_1(\g))
\cong (V_T(\g)\* L_1(\g)))^{\g[t]}$$
by the Frobenius resprocity,
where
$ (V_T(\g)\* L_1(\g)))^{\g[t]}$
denotes the $\g[t]$-invariant subspace of $V_T(\g)\* L_1(\g)$
with respect to the diagonal action.
As $V_T(\g)$, $ L_1(\g)$, and $V_{T+1}(\g)$
are conformal,
$\mc{C}_T(\g)$ is conformal 
with central charge
\begin{align*}
\frac{\tau(K)\dim \g}{\tau(K)+h^{\vee}}+\frac{\dim \g}{h^{\vee}+1}-\frac{(\tau(K)+1)\dim \g}{\tau(K)+h^{\vee}+1}
=\frac{\tau(K)(\tau(K)+2h^{\vee}+1)\dim \g}{(h^{\vee}+1)(\tau(K)+h^{\vee})(\tau(K)+h^{\vee}+1)},
\end{align*}
which equals to
\begin{align}
\frac{\tau(K)(\tau(K)+2h^{\vee}+1)\on{rank} \g}{(\tau(K)+h^{\vee})(\tau(K)+h^{\vee}+1)},
\end{align}
in the case that $\g$ is simply laced.

If $T=\C$ and $\tau(K)=k\in \C$,
we write $\mc{C}_k(\g)$ for $\mc{C}_T(\g)$.
As $V_k(\g)\* L_1(\g)$ belongs to $\on{KL}_{k+1}$,
 Lemma \ref{lem:projective} gives 
 the following assertion.
\begin{Lem}\label{Lem:Cg(g)}
Suppose that $k+h^{\vee}+1\not\in \Q_{\leq 0}$.
Then 
\begin{align*}
\mc{C}_k(\g)\isomap (V_k(\g)\* L_1(\g))_{[0]}.
\end{align*}
\end{Lem}

Throughout this paper we will use the following notation.
\begin{Def} \label{def:RandF}
Let $R$
be the ring of rational functions in $\mathbf{k}$ with poles lying in $\{\mathbb{Q}_{\leq 0}-h^{\vee}\} \cup \{\infty\}$,
regarded as a $\C[K]$-algebra by the structure map $\tau(K)=\mathbf{k}$. Let $F$ denote the quotient field of $R$, which is just field of rational functions $\C(\mathbf{k})$.
\end{Def}

\begin{Pro}\label{Pro:Crg}
We have
$\mc{C}_{R}(\g)\cong (V_R(\g)\* L_1(\g))_{[0]}$. 
Therefore $\mc{C}_{R}(\g)$ is a free $R$-module.
\end{Pro}
\begin{proof}
Although the statement is essentially proved in \cite{CL},
we include the proof for completeness.
By the Frobenius reciprocity 
we have
$\mc{C}_{R}(\g)\cong (V_R(\g)\* L_1(\g))_0^{\g[t]}$.
Hence it is sufficient to show that
$(V_R(\g)\* L_1(\g))_{[0]} =(V_R(\g)\* L_1(\g))_0^{\g[t]}$.
The inclusion
$\supset $ is clear.
Let $v\in (V_R(\g)\* L_1(\g))_{[0]} $.
Then $(L_0v)\* 1=L_0(v\*1)=0$
in $(V_R(\g)\* L_1(\g))\*_R \C_k=V_k(\g)\* L_1(\g)$
for all $k\not \in \Q_{\leq 0}-h^{\vee}$ by Lemma \ref{lem:projective},
where 
$\C_k=R/(\mathbf{k}-k)$.
Thus $L_0v=0$.
Similarly,  $\g[t]v=0$.
\end{proof}

Observe that $\on{KL}_F$ is semisimple (see e.g.\ \cite{Fie06}).
Therefore,
$$\mc{C}_{F}(\g)\cong (V_F(\g)\* L_1(\g))_{[0]}=  (V_F(\g)\* L_1(\g))_0= (V_F(\g)\* L_1(\g))_0^{\g[t]}. $$
It follows from Proposition \ref{Pro:Crg} that 
\begin{align}
\mc{C}_{F}(\g)=\mc{C}_R(\g)\*_R F.
\label{eq:CFg}
\end{align}

\begin{Pro}\label{Pro:character-of-coset}
Suppose that $k+h^{\vee}\not\in \Q_{\leq 0}$.
We have
\begin{align*}
\mc{C}_k(\g)=\mc{C}_R(\g)\*_{R}\C_k.
\end{align*}
In particular,  the character of $\mc{C}_k(\g)$ 
is independent of $k$ and 
coincides with that of $\mc{C}_{F}(\g)$.
\end{Pro}
\begin{proof}
We have
$(V_k(\g)\* L_1(\g))_{[0]}=(V_R(\g)\* L_1(\g))_{[0]}\*_R \C_k$.
Therefore the assertion
follows from 
Lemma \ref{Lem:Cg(g)},
Proposition \ref{Pro:Crg},
and
\eqref{eq:CFg}.
\end{proof}

\section{Wakimoto modules and Screening operators}\label{section:Wakimoto}
 We continue to assume that
$T$ is an integral $\C[K]$-domain
with the structure map $\tau:\C[K]\ra T$
such that $\tau(K)+h^{\vee}$ is invertible.

For $\lam\in \h^*_T$,
the {\em Verma module with highest weight $\lam$ over $T$}
is defined as 
\begin{align*}
\M_T(\lam)=U(\affg)\otimes_{U(\widehat{\mf{b}})}T_{\lam},
\end{align*}
where $T_{\lam}$ denotes the
$\widehat{\mf{b}}$-module $T$ on which
$\widehat{\mf{n}}$ acts trivially,
$\h$ acts by the character $\lam$,
and $K$ acts as the multiplication by $\tau(K)$.
This is an object of the deformed category $\mc{O}_T$,
which 
is
 the category of all  $V_T(\g)$-modules
such that (1) $M$ is a weight module,
(2)  $U(\widehat{\mf{b}})\*_{\C}T. m$ is finitely generated as a $T$-module for all $m\in M$.
We have
$\M_T(\lam)=\bigoplus_{\Delta}\M_T(\lam)_{\Delta}$.
Let 
$ \M_T(\lam)^*\in \mc{O}_T$ be  the contragredient dual 
$\bigoplus_{\mu,\Delta}\on{\Hom}_T(\M_T(\lam)^{\mu}_\Delta,T)$ of $ \M_T(\lam)$.

If $T=\C$ and $\tau(K)=k\in \C\backslash\{-h^{\vee}\}$,
we write $\M_k(\lam)$ for $\M_T(\lam)$
and $\mc{O}_k$ for $\mc{O}_T$.
Let $\L_k(\lam)\in \mc{O}_k$ be the unique simple quotient of $\M_k(\lam)$.
Note that $L_k(\g)\cong \L_k(0)$.

\smallskip
We now introduce the Wakimoto modules following \cite{FeuFre90,Fre05}.
Let $M_{\g}$ be 
 the $\beta\gamma$-system generated
by $a_{\alpha}(z)$,
$a^*_{\alpha}(z)$,
$\alpha\in \Delta$,
satisfying the OPEs
\begin{align*}
a_{\alpha}(z)a^*_{\beta}(w)\sim \frac{\delta_{\alpha \beta}}{z-w},
\quad
a_{\alpha}(z)a_{\beta}(w)\sim a^*_{\alpha}(z)a^*_{\beta}(w)\sim 0.
\end{align*}
Let $\pi_T$ be the Heisenberg vertex algebra over $T$,
which is generated by
fields
$$b_i(z) = \sum_{n\in \mathbb{Z}} (b_i)_{(n)} z^{-n-1},\qquad i=1,\dots, \on{rank}\g,$$
with OPEs
\begin{align}
b_i(z)b_j(w)\sim \frac{\tau(K)(\alpha_i|\alpha_j)}{(z-w)^2}.
\label{eq:b_i}
\end{align}
Define the vertex algebra
$$\Wak{0}{T}:=M_\g\*_{\C} \pi_{T+h^{\vee}}.$$

By \cite[Theorem 5.1]{Fre05},
we have the vertex algebra embedding
\begin{align}
V_T(\g)\hookrightarrow \Wak{0}{T}.
\label{eq:Miura-for-Lie-algebra}
\end{align}
Here, since we work over $T$ we need to replace $\kappa$ in 
\cite[Theorem 5.1]{Fre05} by $$(\tau(k)+h^{\vee})\kappa_0.$$

More generally,
for any
$\lam\in \h^*_T$,
let 
$\pi_{T,\lam}=U(\mc{H})\*_{\mc{H}_{\geq 0}}T_{\lam}$,
where
$\mc{H}=\h[t,t^{-1}]\* \C K\subset \affg$,
$\mc{H}_{\geq 0}=\h[t]\+ \C K \subset \mc{H}$,
$T_{\lam}=T$  on which $\h[t]t$  acts trivially,
$h\in \h$ acts by multiplication by $\lam(h)$,
and $K$ acts by multiplication by $\tau(K)$.
Then $\pi_{T,\lam}$ is naturally a $\pi_T$-module,
and thus,
$$\Wak{\lam}{T}:=M_{\g}\*_\C \pi_{T+h^{\vee},\lam}$$
is a $M_{\g}\*_\C \pi_{T+h^{\vee}}$-module,
and hence,
a
$V_T(\g)$-module,
which belongs to $\mc{O}_T$.
$\Wak{\lam}{T}$ is called the {\em Wakimoto module with highest weight $\lam$ over $T$.}
If $T=\C$ and $\tau(K)=k\in \C$,
$\Wak{\lam}{T}$ is the usual Wakimoto module with highest weight $\lam$
at level $k$
and is denoted also by $\Wak{\lam}{k}$.
If this is the case $\pi_{T+h^{\vee},\lam}$ is denoted  by 
$\pi_{k+h^{\vee},\lam}$.

Let $$L\mf{n}:=\mf{n}[t,t^{-1}]\subset \affg,$$
and 
let
$H^{\frac{\infty}{2}+i}(L\mf{n},M)$
be the semi-infinite $L\mf{n}$-cohomology with coefficients in a smooth $L\mf{n}$-module $M$ (\cite{Feu84}).
Fix a basis $\{x_{\alpha}\mid \alpha\in \Delta_+\}$
of $\mf{n}$,
and let 
$c_{\alpha,\beta}^{\gamma}$ be the corresponding structure constant of $\mf{n}$.
Denote by
$\bigwedge^{\frac{\infty}{2}+\bullet}(\mf{n})$ the fermionic ghost system generated by odd fields
$\psi_{\alpha}(z)$, $\psi_{\alpha}^*(z)$, $\alpha\in \Delta_+$, with OPEs
\begin{align*}
\psi_{\alpha}(z)\psi^*_{\beta}(w)\sim \frac{\delta_{\alpha \beta}}{z-w},
\quad
\psi_{\alpha}(z)\psi_{\beta}(w)\sim \psi^*_{\alpha}(z)\psi^*_{\beta}(w)\sim 0.
\end{align*}
By definition, $H^{\frac{\infty}{2}+\bullet}(L\mf{n},M)$
is the cohomology of the complex 
$(M\* \bigwedge^{\frac{\infty}{2}+\bullet}(\mf{n}),Q^{st}_{(0)})$,
where 
\begin{align*}
Q^{st}(z)=\sum_{n\in \Z}Q^{st}_{(n)}z^{-n-1}=\sum_{\alpha\in \Delta_+}x_{\alpha}(z)\psi_{\alpha}^*(z)-\frac{1}{2}
\sum_{\alpha,\beta,\gamma\in\Delta_+}c_{\alpha,\beta}^{\gamma}\psi_{\alpha}^*(z)\psi_{\beta}^*(z)\psi_{\gamma}(z).
\end{align*}
If $M$ is a $\pi_T$-module, then 
 $H^{\frac{\infty}{2}+\bullet}(L\mf{n},M)$ naturally a $\pi_{T+h^{\vee}}$-module by the correspondence
 \begin{align}
b_i(z)\mapsto b_i(z)+\sum_{\alpha\in \Delta_+}(\alpha|\alpha_i)
:\psi_{\alpha}(z)\psi_{\alpha}^*(z):.
\label{eq:natual-Heisenberg-action}
\end{align}

By construction,
$L\mf{n}$ only acts on the first factor $M_{\g}$ of $\Wak{\lam}{T}=M_{\g}\*_\C \pi_{T+h^{\vee},\lam}$.
Hence
\begin{align}
H^{\frac{\infty}{2}+i}(L\mf{n},\Wak{\lam}{T})\cong 
H^{\frac{\infty}{2}+i}(L\mf{n},M_\g)\* \pi_{T+h^{\vee},\lam}
\cong 
\begin{cases}
\pi_{T+h^{\vee},\lam}&\text{for }i=0\\
0&\text{otherwise,}
\end{cases}
\label{eq:Wakimoto_characterizaion}
\end{align}
by \cite[Theorem 2.1]{Vor93}
since $M_\g$ is free as $U(\mf{n}[t^{-1}]t^{-1})$ and cofree as $U(\mf{n}[t])$-module.
In the case that $\lam=0$,
\eqref{eq:Wakimoto_characterizaion}
gives the isomorphism
$H^{\frac{\infty}{2}+0}(L\mf{n},\Wak{0}{T})\cong \pi_{T+h^{\vee}}$ of vertex algebras,
and \eqref{eq:Wakimoto_characterizaion} is an isomorphism as $\pi_{T+h^{\vee}}$-modules.

\begin{Lem}[\cite{FeiFre92}]\label{Lem:same1}
The above described vertex algebra  isomorphism
\begin{align*}
\pi_{T+h^{\vee}}\ra H^{\frac{\infty}{2}+0}(L\mf{n},\Wak{0}{T}),\quad b_i(z)\mapsto b_i(z)
\end{align*}
coincides with the vertex algebra  homomorphism
$\pi_{T+h^{\vee}}\ra H^{\frac{\infty}{2}+i}(L\mf{n},\Wak{0}{T})$ induced by
the action of  $\pi_{T}\subset V_T(\g)$ on $\Wak{0}{T}$ by \eqref{eq:natual-Heisenberg-action}.
\end{Lem}
\begin{proof}
The difference of the two actions of $b_i(z)$ on $H^{\frac{\infty}{2}+0}(L\mf{n},\Wak{0}{T})$ is 
given by
\begin{align*}
A_i(z)=-\sum_{\alpha\in \Delta_+}(\alpha| \alpha_i^{\vee})
:a_{\alpha}(z)a_{\alpha}^*(z):+\sum_{\alpha\in \Delta_+}(\alpha|\alpha_i^{\vee})
:\psi_{\alpha}(z)\psi_{\alpha}^*(z):,
\end{align*}
see   \cite[Theorem 4.7]{Fre05}.
As $A_i(z)$ commutes with $b_j(z)$ for any $j$,
the corresponding state $A_i=\lim\limits_{z\ra 0}A_i(z)\mathbf{1}$
belongs to the center  $\pi_{T+h^{\vee}}^{\h[t]}$
of  the vertex algebra $\pi_{T+h^{\vee}}=H^{\frac{\infty}{2}+0}(L\mf{n},\Wak{0}{T})$.
On the other hand it is straightforward to see that 
$\pi_{T+h^{\vee}}^{\h[t]}$
is trivial, that is,
$\pi_{T+h^{\vee}}^{\h[t]}=T$.
Hence,  $A_i=0$ in $H^{\frac{\infty}{2}+0}(L\mf{n},\Wak{0}{T})$
as $A_i$ has the conformal weight $1$. 
\end{proof}

Though the following theorem was stated in \cite{FeuFre90} and proved in
 \cite{A-BGG} for $T=\C$, the same proof applies.
\begin{Th}\label{Th:uniqueness-wakimoto}
Assume that $T$ is a field.
The Wakimoto module $\Wak{\lam}{T}$
is a unique object in $\mc{O}_T$
satisfying \eqref{eq:Wakimoto_characterizaion}.
\end{Th}

For $\lam\in \h_T^*$,
set 
\begin{align}
\hat{\lam}=\lam+ \tau(K)\Lam_0\in \affh^*_T. \label{eq:affine-weight}
\end{align}
Let
\begin{align*}
\Delta(\hat \lam)=\{\alpha\in \widehat{\Delta}\mid 
\bra \hat \lam+\hat\rho,\alpha^{\vee}\ket\in \Z\},
\end{align*}
 the set of integral roots with respect to $\hat{\lam}$.

\begin{Pro}\label{Pro:dual-Verma=Wakimoto}
Suppose that $T$ is a field,
$\lam\in \h_T^*$,
and suppose that 
$\widehat{\Delta}(\hat\lam)\cap \{-\alpha+n\delta\mid \alpha\in \Delta_+,\ n\in\Z_{\geq 1}\}=\emptyset $.
Then 
$$\Wak{\lam}{T}\cong  \M_T(\lam)^*.$$
\end{Pro}
\begin{proof}
By \cite[Theorem 3.1]{Ara04},
the assumption implies that 
$\M_T(\lam)$ is free over 
$U(\widehat{\mf{n}}\cap t_{\mu}(\widehat{\mf{n}}_-))$
for any $ \mu\in P_+$,
where $t_{\mu}$ denotes a Tits lifting of $t_{\mu}$.
Since
$\mf{n}_-[t]t=\lim\limits_{\longrightarrow\atop \mu\in P_+}\widehat{\mf{n}}_+\cap t_{\mu}(\widehat{\mf{n}}_-)$,
this shows that $\M_T(\lam)$ is cofree over $U(\mf{n}_-[t]t)$,
and hence,
$\M_T(\lam)^*$ is cofree over $U(\mf{n}[t^{-1}]t^{-1})$.
As $\M_T(\lam)^*$ is obviously free over $U(\mf{n}[t^{-1}]t^{-1})$,
$H^{\frac{\infty}{2}+i}(L\mf{n},\M_T(\lam)^*)=0$ for $i\ne 0$.
It follows from the Euler-Poincar\'{e} principle  that 
the character of $H^{\frac{\infty}{2}+0}(L\mf{n},\M_T(\lam)^*)$  equals to that of $\pi_{T+h^{\vee},\lam}$.
On the other hand, there is an obvious non-zero 
homomorphism 
$\pi_{T+h^{\vee},\lam}\ra H^{\frac{\infty}{2}+0}(L\mf{n},\M_T(\lam)^*)$.
Since  $\pi_{T+h^{\vee},\lam}$ is simple,
we conclude that $H^{\frac{\infty}{2}+0}(L\mf{n},\M_T(\lam)^*)\cong \delta_{i,0}\pi_{T+h^{\vee},\lam}$,
and we are done by Theorem \ref{Th:uniqueness-wakimoto}.
\end{proof}

Let $V(\mf{n})$ be the universal affine vertex algebra
associated with $\mf{n}$,
which can be identified with the
vertex subalgebra
of
$V_T(\g)$ generated by $x_{\alpha}(z)$, $\alpha\in \Delta_+$.
The $L\mf{n}$-action on $M_\g$ induces the vertex algebra
embedding $V(\mf{n})\hookrightarrow M_\g$.

There is also a right action
$x\mapsto x^R$ of $L\mf{n}$ on $M_\g$
that commutes with the left action of $L\mf{n}$ (\cite{Fre05}).
In fact, as a $U(L\mf{n})$-bimodule
$M_{\g}$  is isomorphic to the {\em semi-regular bimodule} \cite{Vor93,Vor99}
of $L\mf{n}$, see \cite{A-BGG}.

\begin{Pro}\label{Pro:key-iso}
\begin{enumerate}
\item (\cite[Proposition 2.1]{A-BGG}). Let $M$ be a $\mf{n}((t))$-module 
that is integrable over $\mf{n}[[t]]$.
There is a $T$-linear isomorphism 
\begin{align*}
\Phi:\Wak{\lam}{T}\*_\C M\overset{\sim}{\ra} \Wak{\lam}{T}\*_\C M
\end{align*}
such that
\begin{align*}
\Phi\circ \Delta(x)=(x\*1)\circ \Phi,
\quad \Phi \circ (x^R\*1)=(x^R\*1-1\* x) \circ \Phi\quad \text{for }x\in L\mf{n}.
\end{align*}
Here 
$\Delta$ denotes the coproduct:
$\Delta(x)=x\*1 +1\*x$.
\item 
Let $V$ be a vertex algebra equipped with a vertex algebra 
 homomorphism
$V(\mf{n})\ra V$,
and the induced 
action of $\mf{n}[[t]]$ on $V$ is integrable.
Then the map $\Phi$ in (1) for $M=V$ is a vertex algebra isomorphism.
\end{enumerate}
\end{Pro}
\begin{proof}
(2) 
The commutative vertex subalgebra
of $M_\g$ generated by
$a_{\alpha}^*(z)$,
$\alpha\in \Delta_+$,
is naturally identified with
functions
$\C[J_{\infty}N]$
on the arc space $J_{\infty}N$
of the unipotent group $N$ whose Lie algebra is $\mf{n}$.
In this identification,
the subalgebra
$\C[N]$ of $\C[J_{\infty}N]$
is identified with
$\C[(a_{\alpha}^*)_{(-1)}]=\C[(a_{\alpha}^*)_{(-1)}]\mathbf{1}$.
The arc space $J_{\infty}N$ is a prounipotent group
whose Lie algebra is $J_{\infty}\mf{n}=\mf{n}[[t]]$.
The vertex subalgebra
$\C[J_{\infty}N]\subset M_\g$
is a  $\mf{n}[[t]]$-bi-submodule
of $M_{\g}$,
and the  $\mf{n}[[t]]$-bi-submodule structure of 
$\C[J_{\infty}N]$ is identical to the one obtained by
differentiating the natural $J_{\infty}N$-bimodule structure of $\C[J_{\infty}N]$.
We have the  isomorphism
\begin{align*}
V(\mf{n})\*_{\C} \C[J_{\infty}N]\overset{\sim}{\ra} M_\g,\quad u\mathbf{1}\* f\mapsto  u f,
\end{align*}
($u\in U(t^{-1}\mf{n}[t^{-1}])$)
as left $t^{-1}\mf{n}[t^{-1}]$-modules and right $\mf{n}[[t]]$-modules.
We have
\begin{align}
x_{\alpha}(z)a_{\beta}^*(w)\sim \frac{1}{z-w}(x_{\alpha}a_\beta^*)(w)
\label{eq:added2019-03-18-10:27}
\end{align}
for $\alpha,\beta\in \Delta_+$,
where on the right-hand side $x_\alpha\in \mf{n}$
acts on $a_{\beta}^*=(a_{\beta}^*)_{(-1)}\mathbf{1}\in \C[N]$ as a left-invariant vector field.

By the assumption the action
of $\mf{n}[[t]]$ on $V$ integrates to the action of
$J_{\infty}N$.
Let
$\phi: V\ra \C[J_{\infty}N]\* V$
be the corresponding comodule map.
Thus,
$\phi(v_i)=\sum_{j}f_{ij}\* v_j$
if $\{v_i\}$ is a basis of $V$
and $g v_i=\sum_{j}f_{ij}(g)v_j$ with $f_{ij}\in \C[J_{\infty}N]$ for all $g\in J_{\infty}N$.
We have
$\phi\circ g=(g\* 1)\circ \phi$ for $g\in J_{\infty}N$.
We shall show that
$\phi$ is a vertex algebra homomorphism,
that is,
\begin{align*}
\phi((v_i)_{(n)}v_k)=\phi(v_i)_{(n)}\phi(v_k)=\sum_{j,l\atop r\geq 0}((f_{ij})_{(-r-1)}f_{kl})\* ((v_{k})_{(n+r)}v_l)
\end{align*}
for all $i,k,n$.
By the definition of $\phi$
this is equivalent to that
\begin{align*}
g.(v_i)_{(n)}v_k=\sum_{i,j\atop r\geq 0}(f_{ij})_{(-r-1)}(g)f_{kl}(g)(v_{k})_{(n+r)}v_l
=\sum_{r\geq 0}(f_{ij})_{(-r-1)}(g) (v_{k})_{(n+r)} (g.v_k),
\end{align*}
for  $g\in J_{\infty}N$,
or equivalently,
\begin{align*}
\on{Ad}(g) (v_i)_{(n)}=\sum_{j\atop r\geq 0}(f_{ij})_{(-r-1)}(g)(v_j)_{(n+r)}
\end{align*}
for  $g\in J_{\infty}N$.
By differentiating both sides,
it is enough to show that
\begin{align}
[x_{(m)},(v_i)_{(n)}]=\sum_{j\atop r\geq 0}(x_{(m)}(f_{ij})_{(-r-1)})(1)(v_j)_{(n+r)}
\label{eq:added-2019-03-17-21}
\end{align}
for $x\in \mf{n}$, $m\geq 0$, $n\in \Z$,
where
$(x_{(m)}(f_{ij})_{(-r-1)})(1)$ is the value of
$x_{(m)}(f_{ij})_{(-r-1)}\in \C[J_{\infty}N]$ at the identity.
By  the commutation formula 
we have
$[x_{(m)},(v_i)_{(n)}]=\sum\limits_{s\geq 0}\begin{pmatrix}m\\s\end{pmatrix}(x_{(s)}v_i)_{(m+n-s)}=
\sum\limits_{s\geq 0\atop j}\begin{pmatrix}m\\s\end{pmatrix}((x_{(s)}f_{ij})(1)v_j)_{(m+n-s)}$.
Hence  \eqref{eq:added-2019-03-17-21}
 follows from the fact that
$$(x_{(m)}(f_{ij})_{(-r-1)})(1)
=\frac{m}{r}(x_{(m-1)}(f_{ij})_{(-r)})(1)=\dots =\begin{pmatrix}m\\r
\end{pmatrix}(x_{(m-r)}(f_{ij})_{(-1)})(1)$$
for $m, r\geq 0$.

Next set
\begin{align*}
\tilde{\phi}:\C[J_{\infty}N]\* V\ra \C[J_{\infty}N]\* V,\quad f\* v\mapsto (f\*1)\phi(v).
\end{align*}
Then
$\tilde{\phi}$ is a linear isomorphism
that satisfies
\begin{align}
\tilde{\phi}\circ (g\* g)=(g\* 1)\circ \tilde{\phi},
\label{eq:added2019-03-18-10:10}\\
\tilde{\phi}\circ (g^R\* 1)=(g^R\* g^{-1})\circ \tilde{\phi},
\label{eq:added2019-03-18-10:10-2}
\end{align}
for $g\in J_{\infty}N$,
where 
$g^R$ denotes the right action $(g^R f)(a)=f(ga)$.
Moreover,
$\tilde{\phi}$ is a vertex algebra homomorphism
since $\phi$ is so and $\C[J_{\infty}N]$ is commutative.

Define the linear isomorphism
\begin{align*}
\Psi:M_\g\* V=V(\mf{n})\* \C[J_{\infty}N]\* V\isomap M_\g\* V,
\quad u\* w\mapsto \Delta(u)(\tilde{\phi}^{-1}(u)),
\end{align*}
($u\in V(\mf{n})$, $w\in \C[J_{\infty}N]\* V$),
where 
$\Delta$ is the coproduct of $ U(t^{-1}\mf{n}[t^{-1}])$ (that is identified with
$V(\mf{n})$)
and
$\C[J_{\infty}N]\* V$ is naturally considered as a vertex subalgebra 
of $M_{\g}\* V$.
We claim that
$\Psi$ is a vertex algebra homomorphism.
To see this,
 first note that
the restrictions of $\Psi$ to vertex subalgebras
$V(\mf{n})$, $\C[J_{\infty}N]\* V$ are 
clearly vertex algebra homomorphism.
Therefore it is sufficient to check that
$\Psi$ preserves the OPE's between generators of
$V(\mf{n})$ and  $\C[J_{\infty}N]\* V$.
By \eqref{eq:added2019-03-18-10:10},
$\Psi( V)=\tilde{\phi}^{-1}(V)$ is contained in the commutant $(M_\g\* V)^{\mf{n}[t]}$
of vertex subalgebra $\Delta(V(\mf{n}))$ in $M_{\g}\* V$.
Also,
since the restriction of $\Psi$ to $ \C[J_{\infty}N]$ is the identify map,
we find that \eqref{eq:added2019-03-18-10:27}
is preserved by 
$\Psi$.
We have shown that $\Psi$ is a vertex algebra isomorphism,
and thus.
$\tilde{\Phi}:=\Psi^{-1}$
is also a vertex algebra isomorphism.

Since $\tilde{\Phi}\circ \Delta(x)=(x\*1)\circ \tilde{\Phi}$ for $x\in \mf{n}\subset V(\mf{n})$
by definition and $\tilde{\Phi}$ is a vertex algebra homomorphism,
we get that
$\tilde{\Phi}\circ \Delta(x)=(x\*1)\circ \tilde{\Phi}$ for all $x\in \mf{n}((t))$.
Next we show that
$ \tilde{\Phi} \circ (x^R\*1)=(x^R\*1-1\* x) \circ \tilde{\Phi}$
for $x\in \mf{n}((t))$.
By the same reasoning as above,
it is sufficient to show that 
the element
$a=\tilde{\Phi}(x^R_{\alpha}\* 1)-x_{\alpha}^R\*1+1\* x_{\alpha}$
is zero for all $\alpha\in\Delta_+$.
By \eqref{eq:added2019-03-18-10:10-2},
we 
have
$a_{(n)}v=0$
for all $n\geq 0$, $v\in M_\g\* V$,
that is,
$a$ belongs to the center of $M_\g\* V$.
Since $M_\g$ is simple,
this implies that $a$ belongs to the center of $V\subset M_\g\* V$.
On the other hand,
we have $x_\alpha^R=x_{\alpha}+\sum_{\beta>\alpha}(P_{\alpha,\beta})_{(-1)}x_{\beta}$
for some polynomial $P_{\alpha,\beta}$
in $\C[N]$
of weight $\alpha-\beta$,
see \cite[Remark 4.4]{Fre05},
where
we count the weight of $a^*_{\alpha}$ as $-\alpha$.
Also we have
$\tilde{\Phi}(x_{\beta}\* 1)=x_{\beta}\* 1-\tilde{\Phi}(1\* x_{\beta})
=x_{\beta}\* 1-\phi(x_{\beta})
=x_{\beta}\* 1-1\* x_{\beta}-\sum_{\gamma>\beta}R_{\beta,\gamma}\* x_{\gamma}$,
where 
$R_{\beta,\gamma}$ is some polynomial in $\C[N]$
of weight $\beta-\gamma$.
It follows that
$a\in \C[N]^* V(\mf{n})\* V$,
where $\C[N]^* $ is the argumentation ideal of $\C[N]$.
Therefore, we get that $a=0$.

The assertion is proved by
extending $\tilde{\Phi}$ to the vertex algebra isomorphism
$\Phi:\Wak{\lam}{T}\*_\C V\overset{\sim}{\ra} \Wak{\lam}{T}\*_\C V$
whose restriction to
 the vertex subalgebra 
 $\pi_{T+h^{\vee}}$
is the identity map.

(1) 
Although the assertion was proved in \cite{A-BGG},
we give a yet another based on the statement (2) we have just proved.
As in (2),  we obtain a linear isomorphism
$\Phi:\Wak{\lam}{T}\*_\C M\overset{\sim}{\ra} \Wak{\lam}{T}\*_\C M$.
Note that  $\Wak{\lam}{T}\*_\C M$ is naturally a module over the vertex algebra
$\Wak{\lam}{T}\*_\C V(\mf{n})$,
and we have the isomorphism
$\Phi:\Wak{\lam}{T}\*_\C V(\mf{n})\isomap \Wak{\lam}{T}\*_\C V(\mf{n})$
of vertex algebras obtained in (2).
By construction we have
$\Phi(u_{(n)})\Phi(v)=\Phi(u_{(n)}v)$ for $u\in \Wak{\lam}{T}\*_\C V(\mf{n})$,
and therefore, $\Phi$ satisfies the required properties. 
\end{proof}

For each $i=1,\dots, \on{rank}\g$,
define
an operator $S_i(z): \Wak{\mu}{T}\ra \Wak{\mu-\alpha_i}{T}$,
$\mu\in \h^*_T$,
by
\begin{align*}
S_{i}(z)= :e_i^R(z):e^{\int -\frac{1}{\tau(K)+h^{\vee}}b_i(z)dz}::,
\end{align*}
where
\begin{align}
&:e^{\int -\frac{1}{\tau(K)+h^{\vee}}b_i(z)dz}:
\label{eq:eb}
\\&=T_{-\alpha_i} z^{-\frac{(b_i)_{(0)}}{\tau(K)+h^{\vee}}}
\exp(-\frac{1}{\tau(K)+h^{\vee}}\sum_{n<0}\frac{(b_i)_{(n)}}{n}z^{-n})\exp(-\frac{1}{\tau(K)+h^{\vee}}\sum_{n>0}\frac{(b_i)_{(n)}}{n}z^{-n}).
\nonumber
\end{align}
Here
$z^{-\frac{(b_i)_{(0)}}{\tau(K)+h^{\vee}}}=\exp(-\frac{(b_i)_{(0)}}{\tau(K)+h^{\vee}}\log z)$
and
 $T_{-\alpha_i}$
is the translation operator
$\pi_{T,\mu}\ra \pi_{T,\mu-\alpha_i}$
sending the highest weight vector to the highest weight vector and commuting with all $(b_j)_{(n)}$, $n\ne 0$.
The residue
\begin{align}
S_{i}:=\int S_{i}(z)dz\label{eq:intertwiner1}
\end{align}
is an
intertwining operator between the $\affg$-modules
$\Wak{0}{T}$ and $\Wak{-\alpha_i}{T}$ (\cite{Fre05}).
\begin{Pro}[\cite{FeiFre92}]\label{Pro:resolution-for-generic-level}
Let $T=F$ (see Definition \ref{def:RandF}), or $T=\C$ with $\tau(K)=k\not\in \mathbb{Q}$.
Then there exists a resolution 
of the $\affg$-module $V_T(\g)$ of the form
\begin{align*}
&0\ra V_T(\g)\ra C_0\overset{d_0}{\ra} C_1\ra\dots \ra C_{\ell(w_0)}\ra 0,\\
& C_i=\bigoplus_{w\in W\atop \ell(w)=i}\Wak{w\circ 0}{T},
\end{align*}
where 
$\ell(w)$ is the length of $w\in W$,
$w_0$ is the longest element of $W$,
and 
$d_0=\oplus_{i=1}^{\on{rank}\g} c_i S_i$
 for some $c_i\in \C^*$.
\end{Pro}
\begin{proof}
By Fiebig's equivalence  \cite{Fie06},
there exists a 
resolution 
\begin{align}
0\ra\L_k(0)\ra C_0\overset{d_0}{\ra} C_1\ra\dots \ra C_n\ra 0
\label{eq:more-general-resoluction}
\end{align}
such that
$C_i=\bigoplus\limits_{w\in W\atop \ell(w)=i}\M_T(w\circ 0)^*$,
which corresponds to the dual of BGG resolution of the trivial representation of $\g$.
By Proposition \ref{Pro:dual-Verma=Wakimoto},
$\M_T(w\circ 0)^*\cong \Wak{w\circ 0}{T}$.
The equality $d_0=\oplus_{i=1}^{\on{rank}\g} c_i S_i$
follows from the facts that  $S_i$ is non-trivial and
that
$\dim_T \Hom_{\mc{O}_T} (\M_T(0)^*,\M_T(-\alpha_i)^*)=1$.
\end{proof}

\section{More on screening operators}\label{section:More}
In this section we let $T=\C$ with $\tau(K)=k\in \C\backslash \{-h^{\vee}\}$.

By the formula
just before Proposition 7.1 of \cite{Fre05},
the construction \eqref{eq:intertwiner1} of the intertwining operator is generalized as follows (see \cite{TsuKan86}
for the details):
Let $\mu\in \h^*$
such that
\begin{align}
(\mu|\alpha_i)+m(k+h^{\vee})=\frac{(\alpha_i|\alpha_i)}{2}(n-1).
\label{eq:weight}
\end{align}
for some $n\in \Z_{\geq 0}$ and $m\in \Z$.
Note that  \eqref{eq:weight}
 is equivalent to
$$\bra \hat{\mu}+\hat{\rho},(\alpha_i+m\delta)^{\vee}\ket=n,$$
where $\hat{\mu}=\mu+k\Lam_0$  as before, see \eqref{eq:affine-weight}.
We have
\begin{align*}
&S_i(z_1)S_i(z_2)\dots S_i(z_{n})|_{\Wak{\mu}{k}}\\
&=\prod_{i=1}^{n}z_i^{-\frac{(\mu|\alpha_i)}{k+h^{\vee}}}
\prod_{1\leq i<j\leq n}(z_i-z_j)^{\frac{(\alpha_i|\alpha_i)}{k+h^{\vee}}}
:S_i(z_1)S_i(z_2)\dots S_i(z_n):.
\end{align*}
By setting
 $z_1=z$,
$z_i=zy_{i-1}$, $i\geq 2$,
we have
\begin{align}
&\prod_{i=1}^{n}z_i^{-\frac{(\mu|\alpha_i)}{k+h^{\vee}}}
\prod_{1\leq i<j\leq n}(z_i-z_j)^{\frac{(\alpha_i|\alpha_i)}{k+h^{\vee}}}
\label{eq:multi-function}\\&=
z^{-\frac{n((\mu|\alpha_i)-\frac{(\alpha_i|\alpha_i)}{2}(n-1))}{k+h^{\vee}}}
\prod_{i=1}^{n-1}y_i^{m-\frac{(n-1)(\alpha_i|\alpha_i)}{2(k+h^{\vee})}}(1-y_i)^{\frac{(\alpha_i|\alpha_i)}{k+h^{\vee}}}\prod_{1\leq i<j\leq n-1}(y_i-y_j)^{\frac{(\alpha_i|\alpha_i)}{k+h^{\vee}}}
\nonumber\\
&=z^{nm}
\prod_{i=1}^{n-1}y_i^{m-\frac{(n-1)(\alpha_i|\alpha_i)}{2(k+h^{\vee})}}(1-y_i)^{\frac{(\alpha_i|\alpha_i)}{k+h^{\vee}}}\prod_{1\leq i<j\leq n-1}(y_i-y_j)^{\frac{(\alpha_i|\alpha_i)}{k+h^{\vee}}}.
\nonumber
\end{align}
Let $\mc{L}^*_{n}(\mu,k)$ be the 
the local system with coefficients in $\C$
associated to the monodromy group of the multi-valued function \eqref{eq:multi-function} 
on the manifold
$Y_{n}=\{(z_1,\dots, z_{n})\in (\C^*)^{n}\mid z_i\neq z_j\}$,
and denote by $\mc{L}_{n}(\mu,k)$  the dual local system of $\mc{L}^*_n(\mu,k)$
(\cite{AomKit11}).
Then, for an element $\Gamma\in H_{n}(Y_{n},\mc{L}_n(\mu,k))$,
\begin{align}
S_i(n,\Gamma):=\int_{\Gamma} S_i(z_1)S_i(z_2)\dots S_i(z_n) dz_1\dots dz_n:\Wak{\mu}{k}\ra \Wak{\mu-n\alpha_i}{k}
\label{eq:intertwiner2}
\end{align}
defines a
$\affg$-module homomorphism.

The following statement was proved by Tsuchiya and Kanie in the case of affine $\mathfrak{sl}_2$ (see \cite[Theorem 0.6]{TsuKan86}), but the same proof applies.
\begin{Th} \label{Th:Tsuhiya-Kanie}
Suppose that
$$\frac{2d(d+1)}{(k+h^{\vee})(\alpha_i|\alpha_i)}\not\in \Z,\quad
\frac{2d(d-n)}{(k+h^{\vee})(\alpha_i|\alpha_i)}\not\in \Z,
$$
for all $1\leq d\leq n-1$.
Then there exits a cycle
$\Gamma\in H_{n}(Y_{n},\mc{L}_n(\mu,k))$ such that  
$S_i(n,\Gamma)$ is non-trivial.
\end{Th}

\begin{Pro}[\cite{Fre92Car}]\label{Pro:resolution-for-generic-level-more-general}
Let $k$ be generic,
$\lam\in P_+$, $\mu \in P_+^{\vee}$.
There exists a resolution 
of the $\affg$-module $\L_k(\lam-(k+h^{\vee})\mu)$ of the form
\begin{align*}
&0\ra \L_k(\lam-(k+h^{\vee})\mu)\ra C_0\overset{d_0}{\ra} C_1\ra\dots \ra C_n\ra 0,\\
& C_i=\bigoplus_{w\in W\atop \ell(w)=i}\Wak{w\circ \lam-(k+h^{\vee})\mu}{k}.
\end{align*}
The map $d_0$ is given by
$$d_0=\sum_{i=1}^{\on{rank}\g}
c_i S_i(n_i,\Gamma_i)$$
for some $c_i\in \C$,
with $n_i=\bra \lam+\rho,\alpha_i^\vee\ket$
and $\Gamma_i$ is the cycle  as in Theorem \ref{Th:Tsuhiya-Kanie}.
\end{Pro}
\begin{proof}
We prove in the same manner as Proposition  \ref{Pro:resolution-for-generic-level}.
Set
$\Lam=\lam-(k+h^{\vee})\mu+k\Lam_0\in \widehat{\h}^*$.
Then
$\hat{\Delta}(\Lam)=t_{-\mu}(\Delta)\cong \Delta$.
Hence
Fiebig's equivalence \cite{Fie06}
implies that  
there exists a 
resolution 
\begin{align}
0\ra \L_k(\lam-(k+h^{\vee})\mu)\ra C_0\overset{d_0}{\ra} C_1\ra\dots \ra C_n\ra 0
\label{eq:more-general-resoluction}
\end{align}
such that
$$C_i=\bigoplus\limits_{w\in W\atop \ell(w)=i}\M_k(w\circ \lam-(k+h^{\vee})\mu)^*.$$
On the other hand,  Proposition \ref{Pro:dual-Verma=Wakimoto} gives  that 
$$\M_k(w\circ \lam-(k+h^{\vee})\mu)^*\cong \Wak{w\circ \lam-(k+h^{\vee})\mu}{k}.$$
The second assertion follows from 
the non-triviality of the map
$S_i(n_i,\Gamma_i)$
and the fact that 
$\Hom_{\affg}(\M_k(\lam-(k+h^{\vee})\mu)^*, \M_k(s_i\circ\lam-(k+h^{\vee})\mu)^*)$ is one-dimensional.
\end{proof}

\begin{Co}[\cite{FeiFre92}]\label{Co:generic-BRST-cohomology}
Let $k$ be generic,
$\lam\in P_+$, $\mu \in P_+^{\vee}$.
Then
\begin{align*}
H^{\frac{\infty}{2}+i}(L\mf{n},\L_k(\lam-(k+h^{\vee})\mu))\cong \bigoplus_{w\in W\atop
\ell(w)=i}\pi^{k+h^{\vee}}_{w\circ \lam-(k+h^{\vee})\mu}
\end{align*}
as modules over the Heisenberg vertex algebra $\pi^{k+h^{\vee}}$.
\end{Co}
\begin{proof}
As Wakimoto modules are acyclic with respect to the cohomology functor $H^{\frac{\infty}{2}+i}(L\mf{n} ,?)$,
$H^{\frac{\infty}{2}+i}(L\mf{n} ,\L_k(\lam-(k+h^{\vee})\mu))$
is the cohomology of the complex obtained by applying the
functor $H^{\frac{\infty}{2}+i}(L\mf{n} ,?)$
to the resolution in
Proposition \ref{Pro:resolution-for-generic-level-more-general},
whence the statement.
\end{proof}

\section{Principal $W$-algebras and the Miura map}\label{section:Miura}
Let
$T$ be an integral $\C[K]$-domain
with the structure map $\tau:\C[K]\ra T$.

Let $f\in \mf{n}_-$ be a principal nilpotent element,
$\{e,f,h\}$ an $\mf{sl}_2$-triple. 
Let $H_{DS}^\bullet(?)$ be the Drinfeld-Sokolov reduction cohomology functor associated to $f$ (\cite{FF90,FreBen04}).
By definition,
$$H_{DS}^\bullet(M)=H^{\frac{\infty}{2}+\bullet}(L\mf{n},M\* \C_{\Psi}),$$
where $L\mf{n}$ acts on $M\* \C_{\Psi}$ diagonally
and
$\C_{\Psi}$ is the one-dimensional representation
of $L\mf{n}$ defined by the character
$\Psi:L\mf{n}\ni xt^n\mapsto \delta_{n,-1}(f|x)$.
The space 
$$\W^T(\g):=H_{DS}^0(V_T(\g))$$
is naturally a vertex algebra,
and is called 
the  {\em principal $W$-algebra associated to
$\g$ over $T$}. 
We have
\begin{align}
H_{DS}^i(V_T(\g))=0\quad  \forall i\ne 0.
\label{eq:vanighing-Fr-Ben}
\end{align}
This was proved in \cite{FreBen04} for $T=\C$,
but the same proof applies for the general case.
We write $\W^k(\g)$ for $\W^T(\g)$ if $T=\C$ and $\tau(K)=k$.

We have
\begin{align}
\gr \W^T(\g)\cong \C[J_{\infty}\mc{S}_f]\* T
\label{eq:associated-graded-W-algebra}
\end{align}
as Poisson vertex algebras,
where $\gr V$ denotes the graded Poisson vertex algebra associated to the canonical filtration
\cite{Li05} of $V$,
$\mc{S}_f=f+\g^e\subset \g=\g^*$ is the Kostant slice,
$J_{\infty}X$ is the infinite jet scheme of $X$.
This was proved in \cite{Ara09b} for $T=\C$, but the same proof applies.
In particular,
$\W^T(\g)$ is free as a $T$-module.

From the proof
of \eqref{eq:vanighing-Fr-Ben},
or the fact
 \eqref{eq:associated-graded-W-algebra},
it follows that
for a given
 $\C[K]$-algebra homomorphism
$T_1\ra T_2$ 
we have
\begin{align*}
\W^{T_2}(\g)=\W^{T_1}(\g)\*_{T_1} {T_2}.
\end{align*}

Suppose that $\tau(K)+h^{\vee}$ is invertible.
Then
the vertex algebra $\W^T(\g)$ is conformal and its central charge is given by
\begin{align}
-\frac{((h+1)(\tau(K)+h^{\vee})-h^{\vee})(r^{\vee}{}^Lh^{\vee}(\tau(K)+h^{\vee})-(h+1))\on{rank}\g}{\tau(K)+h^{\vee}},
\end{align}
where $h$ is the Coxeter number of $\g$, ${}^Lh^{\vee}$ is the dual Coxeter number of the Langlands dual
${}^L\g$ of $\g$,
and $r^{\vee}$ is the maximal number of the edges in the Dynkin diagram of $\g$.
We have
$\W^T(\g)=\bigoplus_{\Delta\in \Z_{\geq 0}}\W^T(\g)_{\Delta}$,
$\W^T(\g)_0=T$,
and each $\W^T(\g)_{\Delta}$ is a free $T$-module of finite rank.
Here $\W^T(\g)_{\Delta}$ is the $T$-submodule of $\W^T(\g)$ spanned by the vectors of conformal weight $\Delta$.

As  explained in \cite{Ara16},
there is a vertex algebra embedding 
\begin{align*}
\Miura: \W^T(\g)\hookrightarrow \pi_{T+h^{\vee}}
\end{align*}
called the {\em Miura map} (\cite{FF90,Ara16}).
The induced map
$\gr \Miura: \gr \W^T(\g)=\C[J_{\infty}\mc{S}_f]\* T\ra \pi_{T+h^{\vee}}=\C[J_{\infty}\h^*]\* T
$ between associated graded Poisson vertex algebras is an injective homomorphism,
and we have
\begin{align}
\gr( \Miura (\W^T(\g)))= (\gr \Miura)(\gr \W^T(\g))=\C[J_{\infty}(\h^*/W)]\* T.
\label{eq:associated-graded-of-Miura}
\end{align}

 \begin{Lem}\label{Lem:base-change-for-Miura}
 Let $T_2$ be a $\C[K]$-subalgebra of $T_1$. 
 Then 
 $$\Miura(\W^{T_2}(\g))=\Miura( \W^{T_1}(\g))\cap \pi_{T_2+h^{\vee}}.$$
\end{Lem}
\begin{proof}
Clearly,
$\Miura(\W^{T_1}(\g))\subset \Miura( \W^{T}(\g))\cap \pi_{T_1+h^{\vee}}$.
Thus,
it is sufficient to show that 
$\gr (\Miura(\W^{T_1}(\g)))=\gr ( \Miura( \W^{T}(\g))\cap \pi_{T_1+h^{\vee}})$.
But this follows from \eqref{eq:associated-graded-of-Miura}.
\end{proof}

The Miura map $\Miura$ may be described as follows.
By applying the functor $H_{DS}^0(?)$ to the embedding 
\eqref{eq:Miura-for-Lie-algebra},
we obtain the vertex algebra homomorphism
\begin{align}
\W^T(\g)=H_{DS}^0(V_T(\g))\ra H_{DS}^0(\Wak{0}{T})\cong \pi_{T+h^{\vee}}.
\label{eq:Miura-2}
\end{align}
Here the last isomorphism follows from the following lemma.
\begin{Lem}[\cite{FeiFre92}]\label{Lem:vanishing-DS-cohomology}
We have $H_{DS}^i(\Wak{0}{T})=0$ for $i\ne 0$
and
$H_{DS}^0(\Wak{0}{T})\cong \pi_{T+h^{\vee}}$ as vertex algebras.
More generally,
$H_{DS}^i(\Wak{\lam}{T})\cong \delta_{i,0}\pi_{T+h^{\vee},\lam}$
as $ \pi_{T+h^{\vee}}$-modules
for any $\lam\in \h^*_T$.
\end{Lem}
\begin{proof}
By applying Proposition \ref{Pro:key-iso}
for $M=\C_{\Psi}$,
we obtain the isomorphism
$$H_{DS}^i(\Wak{\lam}{T})\overset{[\Phi]}{\overset{\sim}{\ra}} H^{\frac{\infty}{2}+i}(L\mf{n},\Wak{\lam}{T})\cong \delta_{i,0}\pi_{T+h^{\vee},\lam},$$
where $[\Phi]$ denotes the map induced
by the isomorphism
$\Phi:\Wak{\lam}{T}\* \C_{\Psi}\overset{\sim}{\ra} \Wak{\lam}{T}\* \C_{\Psi}$ in  Proposition \ref{Pro:key-iso}.
\end{proof}

\begin{Pro}[{\cite{Fre92Car}, see also  \cite[Lemma 5.1]{Genra}}]
The map \eqref{eq:Miura-2} coincides with the Miura map $\Miura$
via the isomorphism
$H_{DS}^0(\Wak{0}{T})\cong \pi_{T+h^{\vee}}$ in Lemma \ref{Lem:vanishing-DS-cohomology}.
\end{Pro}

Since it is a $L\mf{n}$-module homomorphism,
the map
$S_i(z):\Wak{\mu}{T}\ra \Wak{\mu-\alpha_i}{T}$
induces the linear map
$$S_i^W(z) :
\pi_{T+h^{\vee},{\mu}}=H_{DS}^0(\Wak{\mu}{T})\ra\pi_{T+h^{\vee},\mu-\alpha_i}=H_{DS}^0(\Wak{\mu-\alpha_i}{T})$$
for $\mu\in \h_T^*$.
\begin{Lem}\label{Lem:screning-for-W}
For each $i=1,\dots, \on{rank}\g$,
we have
$$S_i^W(z)=:e^{\int -\frac{1}{\tau(K)+h^{\vee}}b_i(z)dz}:,
$$
where the right-hand-side is defined in \eqref{eq:eb}.
\end{Lem}
\begin{proof}
We have used  the isomorphism $\Phi$ in Proposition \ref{Pro:key-iso}
in the proof of Lemma \ref{Lem:vanishing-DS-cohomology},
and we have
$$\Phi\circ  S_i(z)=(S_i(z)+:e^{\int -\frac{1}{\tau(K)+h^{\vee}}b_i(z)dz}:)\circ \Phi.$$
This means that
under the isomorphisms
$H_{DS}^0(\Wak{\mu}{T})\cong H^{\frac{\infty}{2}+i}(L\mf{n},\Wak{\mu}{T})$,
$S_i^W(z)$ is realized as the operator
$$S_i(z)+:e^{\int -\frac{1}{\tau(K)+h^{\vee}}b_i(z)dz}:$$
But 
the weight consideration shows that 
the first factor $S_i(z)$ is the zero map.
\end{proof}

We now reprove the following well-known fact.
\begin{Pro}[\cite{FF90,FreBen04}]
Let $T=F$, or $T=\C$ with $\tau(K)=k\not \in \Q$.
Then
\begin{align*}
\Miura(\W^T(\g))=\bigcap_{i=1}^{\on{rank}\g}(\on{Ker}\int S^W_i(z) dz: \pi_{T+h^{\vee}}\ra \pi_{T+h^{\vee},-\alpha_i}).
\end{align*}
\end{Pro}
\begin{proof}
Consider the exact sequence
$0\ra V_T(\g)\ra \Wak{0}{T}\overset{\oplus_i S_i}{\ra}\bigoplus_i \Wak{-\alpha_i}{T}$
described in
 Proposition \ref{Pro:resolution-for-generic-level}.
By applying the functor $H_{DS}^0(?)$,
we get the exact sequence
$$\W^T(\g)\overset{\Miura}{\ra}\pi_{T+h^{\vee}} \ra \bigoplus_{i}\pi_{T+h^{\vee},-\alpha_i}.$$
Hence by Lemma \ref{Lem:screning-for-W},
$\Miura(\W^T(\g))$ equals to the intersection of the kernel of the maps
$\int S_i^W(z)dz:\pi_{T+h^{\vee}} \ra \pi_{T+h^{\vee},-\alpha_i}$.
\end{proof}

Let $\g^L$ denote the Langlands dual Lie algebra of $\g$.
The Feigin-Frenkel duality states that
\begin{align*}
\W^k(\g)\cong \W^{{}^Lk}({}^L\g)
\end{align*}
for a generic $k$,
where ${}^Lk$ is defined by the formula $r^{\vee}(k+h^{\vee})({}^Lk+{}^Lh^{\vee})=1$ (\cite{FeiFre91}).
\begin{Th}[\cite{FeiFre91,FeiFre96}]\label{Th:FFduality}
The Feigin-Frenkel duality 
$\W^k(\g)\cong \W^{{}^Lk}({}^L\g)$
remains valid for  all non-critical $k$.
\end{Th}
\begin{proof}
Although the statement follows from \cite{FeiFre96}, see \cite[Theorem 6.1]{AgaFreOko},
we include a proof for completeness.
 
For $T=R, F$ (see Definition \ref{def:RandF}),
let ${}^LT$ denote the $\C[K]$-algebra $T$ with structure map
$$K\mapsto \frac{1}{r^{\vee}(\tau(K)+h^{\vee})}-h^{\vee}.$$

We have the isomorphism
\begin{align*}
\pi_{T+h^{\vee}}\isomap {}^L\pi_{{}^L T+h^{\vee}},\quad  
b_i(z)\mapsto -r^{\vee}(\tau(K)+h^{\vee})\frac{(\alpha_i|\alpha_i)}{2}{}^Lb_i(z),
\end{align*}
where
$ {}^L\pi_{{}^L T+h^{\vee}}$ is the Heisenberg vertex algebra
corresponding to ${}^L\g$.
This isomorphism 
restricts to the
the isomorphism
\begin{align*}
\W^F(\g)\cong \W^{{}^LF}({}^L\g),
\end{align*}
see chapter 15 of \cite{FreBen04}. 
By Lemma \ref{Lem:base-change-for-Miura},
this further restricts to the isomorphism
$\W^R(\g)\cong \W^{{}^LR}({}^L\g)$.
Therefore, we get that
$\W^k(\g)=\W^R(\g)\*_R\C_k\cong \W^{{}^LR}({}^L\g)\*_R \C_k=\W^{{}^LR}({}^L\g)\*_{{}^LR}\C_{{}^Lk}=\W^{{}^Lk}(\g^L).$

\end{proof}

\section{Miura map and representation theory of $W$-algebras}
\label{section:Miura}
Let $\on{Zhu}(V)$ be the Zhu's algebra of a vertex operator algebra $V$.
We have the isomorphism
\begin{align*}
\on{Zhu}(\W^T(\g))\overset{\sim}{\ra} \mc{Z}(\g)\* T,
\end{align*}
where $\mc{Z}(\g)$ is the center of $U(\g)$.
This was proved in  \cite{Ara07} for $T=\C$, but the proof given in 
\cite{A2012Dec,Ara16} applies to the general case without any changes.
Below we describe the above isomorphism in terms of the Miura map $\Miura$.

Consider
 the homomorphism of Zhu's algebras
 $$\Miura_Z:\on{Zhu}(\W^T(\g))\ra \on{Zhu}(\pi_{T+h^{\vee}})\cong S(\h_T):=S(\h)\* T$$
induced by  the Miura map $\Miura$.
This further induces the homomorphism
$$\gr \Miura_{Z}:\gr \on{Zhu}(\W^T(\g))\ra \gr \on{Zhu}(\pi_{T+h^{\vee}})=S(\h_T),$$
where 
the associated graded are taken with respect to the filtration defined by Zhu \cite{Zhu96}.
\begin{Lem}\label{Lem:injectivity-Miura-Zhu}
\begin{enumerate}
\item The map $\gr \Miura_{Z}$
is an injective homomorphism onto the subalgebra $S(\h_T)^W=S(\h)^W\* T$,
\item  The map $\Miura_{Z}$ is injective.
\end{enumerate}
\end{Lem}
\begin{proof}
(1) 
Since both $\W^T(\g)$ and $\pi_{T+h^{\vee}}$ admits a PBW basis,
$\gr \Miura_{Z}$ can be identified with the map
$R_{\W^T(\g)}\ra R_{\pi_{T+h^{\vee}}}$ induced by $\Miura$
by \cite[Theorem 4.8]{Ara16},
where $R_V$ denotes the Zhu's $C_2$-algebra of a vertex algebra $V$.
By  \eqref{eq:associated-graded-of-Miura},
this is an embedding on to $S(\h)^W\* T\subset S(\h)\* T=R_{\pi_{T+h^{\vee}}}$.
(2) follows from (1).
\end{proof}

For $\lam\in \h_T^*$,
set 
\begin{align}\label{eq:centralchar}
\chi_{\lam}=\text{(evaluation at $\lam$)}\circ \Miura_Z:\on{Zhu}(\W^T(\g))\ra T.
\end{align}
Also,
define  \begin{align}
\chi=T_{\rho-(\tau(K)+h^{\vee})\rho^{\vee}}\circ \Miura_{Z}:\on{Zhu}(\W^T(\g))\ra S(\h_T),
\label{eq:def:gamma}
\end{align}
where  for $\mu\in \h^*_T$,
$T_{\mu}: S(\h_T) \overset{\sim}{\ra} S(\h_T)$
is the translation defined 
by $(T_{\mu}p)(\lam)=p(\lam-\mu)$,
$\lam\in \h_T^*$.

The following assertion can be regarded as a version of 
the Harish-Chandra isomorphism.
\begin{Pro}\label{Pro:Harish-Chandra}
We have 
$\chi(\on{Zhu}(\W^T(\g)))\subset S(\h)^W\* T$,
and $$\chi:\on{Zhu}(\W^T(\g))\ra S(\h)^W\* T$$ is an algebra isomorphism.
\end{Pro}
\begin{proof}
For $\lam\in \h_T^*$,
let  $\M(\lam)\ra \Wak{\lam}{T}$ be the homomorphism that sends the highest weight vector 
$v_{\lam}\in \M(\lam)$ to the highest weight vector $|\lam\ket $ of $\Wak{\lam}{T}$.
The induced map $H_{DS}^0(\M(\lam))\ra H_{DS}^0(\Wak{\lam}{T})=\pi_{T+h^{\vee},\lam}$
sends  $[v_{\lam}]$  to
 $[|\lam\ket] $,
 where $[v_{\lam}]$ and $[|\lam\ket] $ denote the images of $v_{\lam}$
 and $|\lam\ket$ in $H_{DS}^0(\M(\lam))$ and $\pi_{T+h^{\vee}}$, respectively.
As $z\in \on{Zhu}(\W^T(\g))$ acts on $|\lam\ket $ as the multiplication by
$\chi_{\lam}(z)$,
it follows from \cite[Proposition 9.2.3]{Ara07} that
$\chi_{\lam}(z)=\tilde{\chi}_{\lam-(\tau(K)+h^{\vee})\rho^{\vee}}(\lam)$,
where $\tilde{\chi}_{\lam}=\text{(evaluation at $\lam$)}\circ p$
and $p:\mc{Z}(\g)\* T\ra S(\h_T)$ is the Harish-Chandra projection with respect  to the triangular decomposition
$\g=\mf{n}_-\+ \h\+\mf{n}$.
Hence
$\chi_{\lam}=\chi_{\mu}$ if $\lam+\rho-(\tau(K)+h^{\vee})\rho^{\vee}\in
W(\mu+\rho-(\tau(K)+h^{\vee})\rho^{\vee})$,
and therefore,
$\chi(\on{Zhu}(\W^T(\g)))\subset S(\h)^W\* T$.
Since  $\chi$ is injective 
and $\gr \on{Zhu}(\W^T(\g))\cong S(\h)^W\* T$ 
by Lemma \ref{Lem:injectivity-Miura-Zhu},
$\chi:\on{Zhu}(\W^T(\g))\ra S(\h_T)^W$ must be an isomorphism.
\end{proof}
By Proposition \ref{Pro:Harish-Chandra},
\begin{align}
\chi_{\lam}=\chi_{\mu}
 \iff \lam+\rho-(\tau(K)+h^{\vee})\rho^{\vee}\in
W(\mu+\rho-(\tau(K)+h^{\vee})\rho^{\vee})
\end{align}
for $\lam,\mu\in \h^*_T$.

For $\lam\in \h^*_T$,
let $\mathbf{M}_T(\chi_{\lam})$ be the Verma module of $\W^T(\g)$ with highest weight $\chi_{\lam}$ (see \cite[5.1]{Ara07}).
By definition we have
\begin{align*}
&\mathbf{M}_T(\chi_{\lam})\cong \mathbf{M}_T(\chi_{\mu})\quad\iff \quad 
\chi_{\lam}=\chi_{\mu}\\
& \iff\quad \lam+\rho-(\tau(K)+h^{\vee})\rho^{\vee}\in
W(\mu+\rho-(\tau(K)+h^{\vee})\rho^{\vee}).
\end{align*}
If $\tau(K)+h^{\vee}$ is invertible 
then we have
\begin{align*}
\on{ch} \mathbf{M}_T(\chi_{\lam})=\frac{q^{h_{\lam}}}{\prod_{j=1}^\infty(1-q^j)^{\on{rank}\g}},
\end{align*}
where
\begin{align}
h_{\lam}=\frac{|\lam+\rho|^2-|\rho|^2}{2(\tau(K)+h^{\vee})}-\bra \lam|\rho^{\vee}\ket.
\label{eq:lowest-wt}
\end{align}
Here,
$\on{ch} M=\on{tr}_{M}q^{L_0}$.

If $T=\C$ and $\tau(K)=k$ we write 
$\mathbf{M}_k(\chi_{\lam})$ for $\mathbf{M}_T(\chi_{\lam})$,
and denote by 
$\mathbf{L}_k(\chi)$ be the unique irreducible (graded) quotient of $\mathbf{M}_k(\chi)$\footnote{We have changed slightly the notation for the parametrization of 
simple $\W^k(\g)$-modules from \cite{Ara07}.
Namely we have $\chi_{\lam}=\gamma_{\lam-(k+h^{\vee})\rho^{\vee}}$,
where $\gamma_{\lam}$ is the central character used in \cite{Ara07}.
}.

Let $H_-^{\bullet}(M)$ be the BRST cohomology of the ^^ ^^ $-$"-Drinfeld-Sokolov reduction \cite{FKW92} with coefficient in a $V_k(\g)$-module $M$.
\begin{Th}[\cite{Ara07}]\label{thm:reductionofmodules}${}$
\begin{enumerate}
\item $H^i_-(M)=0$  for all $i\ne 0$ and $M\in \mc{O}_k$.
\item $H^0_-(\M_k({\lam}))\cong \mathbf{M}_k(\chi_{\lam+(k+h^{\vee})\rho^{\vee}})$.
\item $H^0_-(\L_k(\lam)) \cong \begin{cases}   \mathbf{L}_k(\chi_{\lam+(k+h^{\vee})\rho^{\vee}}) & \on{if} \ \bra \lam+\rho,\alpha^{\vee}\ket
\not \in \N  \on{for\ any} \alpha\in \Delta_+ \\
0& \on{else} \end{cases}$
\end{enumerate}
\end{Th}

\begin{Th}[{\cite[Theorem 9.1.4]{Ara07}}]\label{Th:+-and-}
Suppose that $\bra\hat \lam+\hat{\rho},\alpha^{\vee}\ket \not \in \Z$ for all
$\alpha\in \{-\beta+n\delta\mid \beta\in \Delta_+, 1\leq n\leq \on{ht}(\beta)\}$.
Then
$H^0_{DS}(\L_k( \lam))\cong \mathbf{L}_k(\chi_{\lam})$.
\end{Th}
\section{More on screening  operators for $W$-algebras}
In this section we let $T=\C$ with $\tau(K)=k\in \C\backslash \{-h^{\vee}\}$.
As in Section \ref{section:More},
let $\mu\in \h^*$
such that
\begin{align}
(\mu|\alpha_i)+m(k+h^{\vee})=\frac{(\alpha_i|\alpha_i)}{2}(n-1).
\label{eq:weight2}
\end{align}
for some $n\in \Z_{\geq 0}$ and $m\in \Z$.
For an element $\Gamma\in H_{n}(Y_{n},\mc{L}_n(\mu,k))$,
set
\begin{align}
S_i^W(n,\Gamma):=\int_{\Gamma} S_i^W(z_1)S_i^W(z_2)\dots S_i^W(z_n) dz_1\dots dz_n:
\pi_{k,\mu}\ra \pi_{k,\mu-n\alpha_i}.
\label{eq:intertwiner2W}
\end{align}

The following result was proved in 
\cite{TsuKan86} for the case that $\g=\mf{sl}_2$.
\begin{Th}\label{Th:Tsuhiya-Kanie2}
The map $S_i^W(n,\Gamma)$ is the  homomorphism of $\W^k(\g)$-modules
obtained from the $\affg$-homomorphism
 $S_i(n,\Gamma):\Wak{\mu}{k}\ra \Wak{\mu-n\alpha_i}{k}$
 by
 applying the functor $H_{DS}^0(?)$.
Moreover,
suppose that
$$\frac{2d(d+1)}{(k+h^{\vee})(\alpha_i|\alpha_i)}\not\in \Z,\quad
\frac{2d(d-n)}{(k+h^{\vee})(\alpha_i|\alpha_i)}\not\in \Z,
$$
for all $1\leq d\leq n-1$.
Then there exits a cycle
$\Gamma\in H_{n}(Y_{n},\mc{L}_n(\mu,k))$ such that  
$S_i^W(n,\Gamma)$ is non-trivial.
\end{Th}
\begin{proof}
The first assertion following from
 Lemma \ref{Lem:screning-for-W}.
 The proof of the non-triviality is the same as 
 that of Theorem \ref{Th:Tsuhiya-Kanie} (and hence as \cite{TsuKan86}).
\end{proof}

\begin{Pro}\label{Pro:W-gen-Weyl}
Let $k$ be generic,
$\lam\in P_+$,
$\mu\in P_+^{\vee}$.
Then
\begin{align*}
&\mathbf{L}_k(\chi_{\lam-(k+h^{\vee})\mu})
\cong \bigcap_{i=1}^r \on{Ker}
S_i^W(n_i,\Gamma_i): \pi_{\lam-(k+h^{\vee})\mu}^{k+h^{\vee}}\ra\pi_{\lam-(k+h^{\vee})\mu-n_i\alpha_i}^{k+h^{\vee}}
\end{align*}
where $n_i=\bra \lam+\rho,\alpha_i^\vee\ket$
and $\Gamma_i\in H_{n}(Y_{n},\mc{L}_n(\lam-(k+h^{\vee})\mu,k))$ is as in Theorem \ref{Th:Tsuhiya-Kanie2}.
\end{Pro}
\begin{proof}
By Theorem \ref{Th:+-and-},
we have 
$\mathbf{L}_k(\chi_{\lam-(k+h^{\vee})\mu})\cong H_{DS}^0(\L_k(\lam-(k+h^{\vee})\mu))$.
As Wakimoto modules are acyclic with respect to the cohomology functor $H_{DS}^\bullet(?)$,
$H^{\bullet}_{DS}(\L_k(\lam-(k+h^{\vee})\mu))$
is the cohomology of the complex obtained by applying the
functor $H_{DS}^\bullet(?)$
to the resolution in
Proposition \ref{Pro:resolution-for-generic-level-more-general}.
Therefore the assertion follows from Theorem \ref{Th:Tsuhiya-Kanie2}.
\end{proof}

\section{Coset construction of universal principal $W$-algebras}\label{section:coset-vs-W}
We assume that $\g$ is simply laced for the rest of the paper.
The goal of this section is to provide a proof of Main Theorem \ref{Main2}.

For $m\in \N$,
let $P^m_+=\{\lam\in P_+\mid \lam(\theta)\leq m\}$.
Then,
$\{\L_k(\lam)\mid \lam \in P^m_+\}$
gives  the complete set of isomorphism classes of irreducible integrable
representation of  $\affg$ of level $m$.

Recall that, 
for $\nu\in P_+^1$,
the Frenkel-Kac construction \cite{FreKac80} gives the isomorphism
of $\affg$-modules
\begin{align}
\L_1(\nu)\cong \bigoplus_{\beta\in Q}\pi_{1,\nu+\beta},
\label{eq:Frenkel-Kac-moules}
\end{align}
where the action of $\affg$ on the right-hand-side is given by
\begin{align*}
h_i(z)\mapsto \alpha_i(z),
\quad e_i(z)=:e^{\int \alpha_i(z)dz}:,
\quad f_i(z)=:e^{-\int \alpha_i(z)dz}:.
\end{align*}
Here we denoted by $\alpha_i(z)$  the generator of the Heisenberg vertex algebra $\pi_{1}$ satisfying the OPE
\begin{align}
\alpha_i(z)\alpha_j(w)\sim \frac{(\alpha_i,\alpha_j)}{(z-w)^2}.
\label{eq:alpha_i}
\end{align}

Let $T$ be an integral $\C[K]$-domain
with the structure map $\tau:\C[K]\ra T$
such that $\tau(K)+h^{\vee}$ and $\tau(K)+h^{\vee}+1$ are invertible.
Consider 
 the vertex algebra
$H^{\frac{\infty}{2}+\bullet}(L\mf{n} ,\Wak{0}{T}\* L_1(\g))$,
where
$L\mf{n} $ acts on $\Wak{0}{T}\* L_1(\g)$ diagonally.
For $\lam\in \h^*_T$ and $\nu\in P_+^1$,
$H^{\frac{\infty}{2}+i}(L\mf{n},\Wak{\lam}{T}\* \L_1(\nu))$ is naturally a 
$H^{\frac{\infty}{2}+\bullet}(L\mf{n} ,\Wak{0}{T}\* L_1(\g))$-module.
\begin{Lem}\label{Lem:decomp1}
There is an embedding 
$$\pi_{T+h^{\vee}}\* \pi_{1}\hookrightarrow H^{\frac{\infty}{2}+0}(L\mf{n},\Wak{0}{T}\* L_1(\g))$$
of vertex algebras,
and we have
$$H^{\frac{\infty}{2}+i}(L\mf{n},\Wak{\lam}{T}\* \L_1(\nu))\cong 
\delta_{i,0} \bigoplus_{\beta}\pi_{T+h^{\vee},\lam}\* \pi_{1,\nu+\beta}$$
as $\pi_{T+h^{\vee}}\* \pi_{1}$-modules
for $\lam\in \h^*_T$, $\nu\in P_+^1$.
\end{Lem}
\begin{proof}
By Proposition \ref{Pro:key-iso},
we have  the isomorphism \begin{align*}
H^{\frac{\infty}{2}+0}(L\mf{n},\Wak{0}{T}\* L_1(\g))\overset{[\Phi]}{\overset{\sim}{\ra}}
H^{\frac{\infty}{2}+0}(L\mf{n},\Wak{0}{T})\* L_1(\g)\cong \pi_{T+h^{\vee}}\* L_1(\g)\\
\cong \bigoplus_{\beta\in Q}\pi_{T+h^{\vee}}\* \pi_{1,\beta},
\end{align*}
of vertex algebras,
where $[\Phi]$ denotes the map induced
by the isomorphism
$\Phi:\Wak{0}{T}\* L_1(\g) \overset{\sim}{\ra} \Wak{0}{T}\* L_1(\g)$ 
for $M=L_1(\g)$ in  Proposition \ref{Pro:key-iso}.
Similarly,
we have
\begin{align*}
H^{\frac{\infty}{2}+0}(L\mf{n},\Wak{\lam}{T}\* \L_1(\nu))\overset{[\Phi]}{\overset{\sim}{\ra}}
H^{\frac{\infty}{2}+0}(L\mf{n},\Wak{\lam}{T})\* \L_1(\nu)\cong \pi_{T+h^{\vee},\lam}\* \L_1(\nu)\\
\cong \bigoplus_{\beta\in Q}\pi_{T+h^{\vee},\lam}\* \pi_{1,\nu+\beta}
\end{align*}
as  $\pi_{T+h^{\vee}}\* \pi_{1}$-modules.
\end{proof}

The proof of the following assertion is the same as Lemma \ref{Lem:same1}.
\begin{Lem}\label{Lem:same2}
The embedding
$\pi_{1}\hookrightarrow H^{\frac{\infty}{2}+0}(L\mf{n},\Wak{0}{T}\* L_1(\g))$
in Lemma \ref{Lem:decomp1}
coincides with 
the
 vertex algebra homomorphism 
\begin{align*}
\begin{array}{ccc}
\pi_{1}& \ra & H^{\frac{\infty}{2}+0}(L\mf{n},\Wak{0}{T}\* L_1(\g)),\\
\alpha_i(z)&\mapsto & \alpha_i(z)-\sum_{\alpha\in \Delta_+}\alpha(\alpha_i^{\vee})
:a_{\alpha}(z)a_{\alpha}^*(z):+\sum_{\alpha\in \Delta_+}\alpha(\alpha_i^{\vee})
:\psi_{\alpha}(z)\psi_{\alpha}^*(z):.
\end{array}
\end{align*}
Here we have omitted the tensor product sign.
\end{Lem}

Let $T'$ denote the $\C[K]$-algebra
$T$ with the structure map 
\begin{align}
\label{eq:tprime} K\mapsto \tau'(K):=\frac{\tau(K)+h^{\vee}}{\tau(K)+h^{\vee}+1}-h^{\vee}.\end{align}
We have the vertex algebra isomorohism
\begin{align}
\begin{array}{ccc}
 \pi_{T+h^{\vee}+1}\* \pi_{T'} & \isomap & \pi_{T+h^{\vee}}\* \pi_{1},\\
 b_i(z)\* 1 & \mapsto & b_i(z)\* 1+1\* \alpha_i(z),\\
 1\* b_i(z) &\mapsto  &
 \frac{1}{\tau(K)+h^{\vee}+1}b_i(z)\* 1 -\frac{\tau(K)+h^{\vee}}{\tau(K)+h^{\vee}+1} 1\* \alpha_i(z).
\end{array}
\label{eq:commutant-embedding}
\end{align}
Note that,
by Lemma \ref{Lem:same2},
the map
$$ \pi_{T+h^{\vee}+1}\hookrightarrow  \pi_{T+h^{\vee}+1}\* \pi_{T'} \isomap  \pi_{T+h^{\vee}}\* \pi_{1}\hookrightarrow 
H^{\frac{\infty}{2}+0}(L\mf{n} ,\Wak{0}{T}\* L_1(\g)),$$
where the first map is 
the embedding to the first component,
coincides with the vertex algebra homomorphism  induced by the diagonal embedding  
$\pi_{T+1}\hookrightarrow \Wak{0}{T}\* L_1(\g)$.

The following assertion follows immediately from 
Lemma \ref{Lem:decomp1}.
\begin{Pro}\label{Pro:com-semiinf-Wakimoto}
We have the isomorphism
\begin{align*}
H^{\frac{\infty}{2}+i}(L\mf{n} ,\Wak{\mu}{T}\* \L_1(\nu))\cong 
\delta_{i,0}\bigoplus_{\lam\in \h^*\atop \lam-\mu-\nu\in Q}\pi_{T+h^{\vee}+1,\lam}\otimes \pi_{T'+h^{\vee}, \mu-(\tau'(K)+h^{\vee})\lam}
\end{align*}
as $ \pi_{T+h^{\vee}+1}\* \pi_{T'+h^{\vee}}$-modules
for $\lam\in \h^*_T$ and $\nu\in P_+^1$.
In particular,
$ \pi_{T+h^{\vee}+1}$ and
$\pi_{T'+h^{\vee}}$ form a dual pair in $H^{\frac{\infty}{2}+0}(L\mf{n},\Wak{0}{T}\* L_1(\g))$.
\end{Pro}

Now recall the commutant vertex algebra
$\mc{C}_T(\g)=\on{Com}(V_{T+1}(\g), V_T(\g)\* L_1(\g))$.
We define the vertex algebra homomorphism
\begin{align*}
\uppsi_T:\mc{C}_T(\g)\ra  \pi_{T'+h^{\vee}}\end{align*}
as follows.
First, define the vertex algebra homomorphism
\begin{align}  \label{eq:uppsi1T} 
\uppsi_{1,T}:\mc{C}_T(\g)= (V_T(\g)\* L_1(\g))^{\g[t]}\ra 
H^{\frac{\infty}{2}+0}(L\mf{n}, V_T(\g)\* L_1(\g)),
\quad v\mapsto [v\* \mathbf{1}].
\end{align} 
The diagonal  embedding
 $ \pi_{T+1}\hookrightarrow V_T(\g)\* L_1(\g)$,
gives rise to
 a  vertex algebra embedding
$
 \pi_{T+h^{\vee}+1}\hookrightarrow H^{\frac{\infty}{2}+0}(L\mf{n} ,V_T(\g)\* L_1(\g))
$.
By definition, the image of $\uppsi_{1,T}$ is contained in 
$$\on{Com}(\pi_{T+h^{\vee}+1},H^{\frac{\infty}{2}+0}(L\mf{n},V_T(\g)\* L_1(\g))) \subset H^{\frac{\infty}{2}+0}(L\mf{n},V_T(\g)\* L_1(\g)).$$
Next
let 
\begin{equation}  \begin{split} \label{eq:uppsi2T} 
\uppsi_{2,T}:\on{Com}(\pi_{T+h^{\vee}+1}, H^{\frac{\infty}{2}+0}&(L\mf{n},V_T(\g)\* L_1(\g)))\\
\ra & \on{Com}(\pi_{T+h^{\vee}+1}, H^{\frac{\infty}{2}+0}(L\mf{n},\Wak{0}{T}\* L_1(\g)))
\cong \pi_{T'+h^{\vee}}
\end{split} \end{equation}
be the vertex algebra homomorphism
induced by 
the embedding $V_T(\g)\hookrightarrow \Wak{0}{T}$.
We set  
\begin{align} \label{eq:uppsiT} 
\uppsi_T=\uppsi_{2,T}\circ \uppsi_{1,T}.
\end{align}

If $T=\C$ and $\tau(K)=k$ we write
$\uppsi_{k}$ for $\uppsi_{T}$
and $\uppsi_{i,k}$ for $\uppsi_{i,T}$, $i=1,2$,
and we define the complex number $\ell$ by 
the formula \eqref{eq:ellintro}
\begin{Pro}
Suppose that $k+h^{\vee}\not\in \Q_{\leq 0}$.
The  vertex algebra homomorphism $$\uppsi_k:\mc{C}_k(\g)\ra \pi_{\ell+h^{\vee}}$$ is injective.
\end{Pro}
\begin{proof}
We show both 
$\uppsi_{1,k}$ and $\uppsi_{2,k}$ are injective.

The vertex algebra
$H^{\frac{\infty}{2}+0}(L\mf{n},V_{k+1}(\g))$ is naturally conformal.
For a $H^{\frac{\infty}{2}+0}(L\mf{n},V_{k+1}(\g))$-module $M$
and $\Delta\in \C$,
set 
$M_{[\Delta]}=\{m\in M\mid (L_0-\Delta)^r m=0,\  r\gg 0\}$.

Let $N$ be the cokernel of the injection
$$V_{k+1}(\g)\* \mc{C}_k(\g)\hookrightarrow V_k(\g)\* L_1(\g),\quad
u\mathbf{1}\* v\mapsto \Delta(u)v,\quad u\in U(\affg).$$
The long exact sequence of the BRST cohomology gives 
the exact sequence
\begin{align}
H^{\frac{\infty}{2}-1}(L\mf{n},N)\ra H^{\frac{\infty}{2}+0}(L\mf{n},V_{k+1}(\g)\* \mc{C}_k(\g))
\ra
H^{\frac{\infty}{2}+0}(L\mf{n}, V_{k}(\g)\* L_1(\g)).
\label{eq:1}
\end{align}
Since 
 $N_{[\Delta]}=0$ for $\Delta\in \Z_{\leq 0}$ by Lemma \ref{Lem:Cg(g)},
we have
$H^{\frac{\infty}{2}-1}(L\mf{n},N)_{[0]}=0$.
Hence
\eqref{eq:1}
restricts to the embedding
\begin{align}
H^{\frac{\infty}{2}+0}(L\mf{n},V_{k+1}(\g)\* \mc{C}_k(\g))_{[0]}
\hookrightarrow 
H^{\frac{\infty}{2}+0}(L\mf{n}, V_{k}(\g)\* L_1(\g))_{[0]}.
\label{eq:embedding-1}
\end{align}
On the other hand, $H^{\frac{\infty}{2}+0}(L\mf{n},V_{k+1}(\g)\* \mc{C}_k(\g))_{[0]}=
H^{\frac{\infty}{2}+0}(L\mf{n},V_{k+1}(\g))_{[0]}\* \mc{C}_k(\g)=\C\*\mc{C}_k(\g)$.
It follows that \eqref{eq:embedding-1} coincides with $\uppsi_{1,k}$.
In particular,
$\uppsi_{1,k}$ is injective.

Next, 
let $U$ be the cokernel of the embedding $V_k(\g)\hookrightarrow \Wak{0}{k}$.
This gives rise to the exact sequence
$$H^{\frac{\infty}{2}-1}(L\mf{n},U\* L_1(\g))\ra H^{\frac{\infty}{2}+0}(L\mf{n},V_k(\g)\* L_1(\g))\ra
H^{\frac{\infty}{2}+0}(L\mf{n},\Wak{0}{k}\* L_1(\g)).$$
In the same way as above we have
$H^{\frac{\infty}{2}-1}(L\mf{n},U\* L_1(\g))_{[0]}=0$, 
and so 
the above map restricts to the embedding
\begin{align}
H^{\frac{\infty}{2}+0}(L\mf{n},V_k(\g)\* L_1(\g))_{[0]}\hookrightarrow 
H^{\frac{\infty}{2}+0}(L\mf{n},\Wak{0}{k}\* L_1(\g))_{[0]}.
\label{eq:added2019.03.17}
\end{align}
Because
$\on{Com}(\pi_{T+h^{\vee}+1}, H^{\frac{\infty}{2}+0}(L\mf{n},V_T(\g)\* L_1(\g))) \subset H^{\frac{\infty}{2}+0}(L\mf{n},V_k(\g)\* L_1(\g))_{[0]}$,
$ \on{Com}(\pi_{T+h^{\vee}+1}, H^{\frac{\infty}{2}+0}(L\mf{n},\Wak{0}{T}\* L_1(\g)))
\subset H^{\frac{\infty}{2}+0}(L\mf{n},\Wak{0}{k}\* L_1(\g))_{[0]}$
and $\uppsi_{2,k}$ is the restriction
of \eqref{eq:added2019.03.17},
it follows that
$\uppsi_{2,k}$ is injective as well.
\end{proof}

Recall the $\C[K]$-algebras $R$ and $F$ defined in Section \ref{section:universal-coset} (see Definition \ref{def:RandF}), and the Miura map $\Miura$ in Section \ref{section:Miura}. Also, recall that $R'$ and $F'$ denote the $\C[K]$-algebras obtained by modifying the structure map on $R$ and $F$ using \eqref{eq:tprime}.
\begin{Th}\label{Th:iso}
The map $\uppsi_F:\mc{C}_F(\g)\ra \pi_{F'+h^{\vee}}$ is an injective vertex algebra homomorphism onto 
$\Miura(\W^{F'}(\g))$.
Therefore,
$\mc{C}_F(\g)\cong \W^{F'}(\g)$ as vertex algebras.
\end{Th}
The proof of Theorem \ref{Th:iso} will be given in Section \ref{Section:proof-of-generic-iso}.

Theorem \ref{Th:iso}
together with 
Proposition \ref{Pro:character-of-coset}
 immediately gives  the following. 
\begin{Co}\label{Co:character-are-the-same}
Assume that $k+h^{\vee}\not \in \Q_{\leq 0}$.
Then
$\on{ch}\mc{C}_k(\g)=\on{ch}\W^{\ell}(\g)$.
\end{Co}

\begin{Th}\label{Th:universalGKO}
\begin{enumerate}
\item The map $\uppsi_R:\mc{C}_R(\g)\ra \pi_{R'+h^{\vee}}$ is an injective vertex algebra homomorphism onto 
$\Miura(\W^{R'}(\g))$.
Therefore,
$\mc{C}_R(\g)\cong \W^{R'}(\g)$ as vertex algebras.
\item Assume that $k+h^{\vee}\not \in \Q_{\leq 0}$.
The map $\uppsi_k:\mc{C}_k(\g)\ra \pi_{\ell+h^{\vee}}$ is an injective vertex algebra homomorphism onto 
$\Miura(\W^{\ell}(\g))$.
Therefore,
$\mc{C}_k(\g)\cong \W^{\ell}(\g)$ as vertex algebras.
\end{enumerate}
\end{Th}
\begin{proof}
Consider the vertex algebra homomorphism
$\uppsi_R:\mc{C}_R(\g)\ra \pi_{R'}$.
By Proposition \ref{Pro:Crg}
and  \eqref{eq:CFg},
 $\mc{C}_R(\g)\subset \mc{C}_F(\g)$,
 and so $\uppsi_R$ is the  restriction of $\uppsi_F$ to $\mc{C}_R(\g)$.
 Hence,
 $$\uppsi_R(\mc{C}_R(\g))\subset \uppsi_F(\mc{C}_F(\g))\cap \pi_{R'}= \Miura(\W^{F'}(\g))\cap \pi_{R'}=\Miura(\W^{R'}(\g))$$
 by Lemma \ref{Lem:base-change-for-Miura}.
Since 
$\mc{C}_k(\g)=\mc{C}_R(\g)\*_R\C_k$
by Proposition \ref{Pro:character-of-coset},
this implies that
\begin{align*}
\uppsi_k(\mc{C}_k(\g))
=\on{Im}(\uppsi_R(\mc{C}_R(\g))\*_R \C_k\ra \pi_{R'+h^{\vee}}\*_R \C_k=\pi_{\ell+h^{\vee}} )\\
\subset \Miura(\W^{R'}(\g))\*_{R'} \C_{\ell}= \Miura(\W^{\ell}(\g)).
\end{align*}
Since $\uppsi_k(\mc{C}_k(\g))$ and $\Miura(\W^{\ell}(\g))$
 have the same character by Corollary \ref{Co:character-are-the-same},
we conclude that 
the natural map
$\uppsi_R(\mc{C}_R(\g))\*_R \C_k\ra \pi_{R'+h^{\vee}}\*_R \C_k$
is injective and  that
$\uppsi_k(\mc{C}_k(\g))=\Miura(\W^{\ell}(\g))$.
As $\uppsi_R(\mc{C}_R(\g))\*_R \C_k
= \Miura(\W^{R'}(\g))\*_{R'} \C_{\ell}$
for all $k\in \on{Specm}{R}$,
we get that 
$\uppsi_R(\mc{C}_R(\g))= \Miura(\W^{R'}(\g))$.
\end{proof}

\begin{Rem}
By the Feigin-Frenkel duality,
Theorem \ref{Th:universalGKO}
implies that 
\begin{align*}
(V_k(\g)\* L_1(\g))^{\g[t]}\cong (V_{-k-2h^{\vee}-1}(\g)\* L_1(\g))^{\g[t]}
\end{align*}
for a generic $k$.
\end{Rem}

\section{Key Observation}
In this section we establish the key fact 
needed 
for the proof of  Theorem \ref{Th:iso}.

Let $T$ be an integral $\C[K]$-domain
with the structure map $\tau:\C[K]\ra T$
such that $\tau(K)+h^{\vee}$ is  invertible.

Recall the operator $S_i(z): \Wak{\mu}{T}\ra \Wak{\mu-\alpha_i}{T}$ defined in Section \ref{section:Wakimoto}.
As $S_i(z)$ is an $\wh{\h}\+L\mf{n}$-module homomorphism,
$S_i(z)\* \on{id}_{\L_1(\nu)}$
defines an $\wh{\h}\+L\mf{n}$-module homomorphism
$\Wak{\mu}{T}\* \L_1(\nu)\ra \Wak{\mu-\alpha_i}{T}\* \L_1(\nu)$.
Therefore,
it induces the
$T$-linear map 
$$[S_i(z)\* \on{id}_{\L_1(\nu)}]:H^{\frac{\infty}{2}+\bullet}(L\mf{n} ,\Wak{\mu}{T}\* \L_1(\nu))\ra
H^{\frac{\infty}{2}+\bullet}(L\mf{n} ,\Wak{\mu-\alpha_i}{T}\* \L_1(\nu)).$$

\begin{Pro}\label{Pro:key}
Let  $\nu\in P_+^1$.
Under the isomorphism
in Proposition \ref{Pro:com-semiinf-Wakimoto},
we have
\begin{align*}
[S_i(z)\* \on{id}_{\L_1(\nu)}]= -1\* {S}_i^W(z):
&\bigoplus_{\lam\in \h_T^*\atop
\lam-\mu-\nu\in Q}\pi_{T+h^{\vee}+1,\lam}\otimes \pi_{T'+h^{\vee},\mu-(\tau'(K)+h^{\vee})\lam}
\\ &\ra 
\bigoplus_{\lam\in \h_T^*\atop
\lam-\mu-\nu\in Q}\pi_{T+h^{\vee}+1,\lam}\otimes \pi_{T'+h^{\vee},\mu-\alpha_i-(\tau'(K)+h^{\vee})\lam}.
\end{align*}
\end{Pro}
\begin{proof}
We have used the map
$\Phi$ in 
Proposition \ref{Pro:key-iso}
to achieve the isomorphism
in Proposition \ref{Pro:com-semiinf-Wakimoto},
and so the factor 
$e_i(z)^R$ of $S_i(z)$ 
becomes
$e_i^R(z)$ minus the action of $e_i(z)$ on $\L_1(\nu)$.
Namely,
\begin{align*}
\Phi\circ S_i(z)\* \on{id}_{\L_1(\nu)}&= (e_i^R(z)\*1 -1\*e_i(z)):e^{\int -\frac{1}{\tau(k)+h^{\vee}}b_i(z)}dz:
\circ \Phi\\
&=S_i(z)\* 1 \circ \Phi-
:e^{\int -\left(\frac{1}{\tau(K)+h^{\vee}}b_i(z)-\alpha_i(z)\right)}dz:\circ \Phi.
\end{align*}
Here,
as before,
$b_i(z)$ and $\alpha_i(z)$ are 
generating fields  of 
$\pi_{T+h^{\vee}}\subset \Wak{0}{T}$ and 
$\pi_1\subset L_1(\g)$
described in 
\eqref{eq:b_i} and \eqref{eq:alpha_i},
respectively,
and we have omitted the tensor product sign.
The first factor  $S_i(z)\* 1$
is the zero map by weight consideration.
The second factor equals to 
$-1\* S_i^W(z)$
since
\begin{align*}
&\frac{1}{\tau(K)+h^{\vee}}b_i(z)-\alpha_i(z)\\&=
\frac{1}{\tau'(K)+h^{\vee}}\left(\frac{1}{\tau(K)+h^{\vee}+1}b_i(z)-
\frac{\tau(K)+h^{\vee}}{\tau(K)+h^{\vee}+1}\alpha_i(z)\right),
\end{align*}
see 
\eqref{eq:commutant-embedding}
and Lemma \ref{Lem:screning-for-W}.
This completes the proof.
\end{proof}

\section{Proof of Theorem \ref{Th:iso}}
\label{Section:proof-of-generic-iso}
Let $T$ be the field $F$,
or $\C$ with the structure map $\tau(K)=k \ne \Q$,
so that $\on{KL}_k$ is semisimple.
We will prove that 
$\uppsi_T$ gives the isomorphism
\begin{align}
\uppsi_T:\mc{C}_T(\g)\overset{\sim}{\ra} \Miura(\W^{T'}(\g))\subset  \pi_{T'+h^{\vee}}.
\end{align}
Here $\uppsi_T=\uppsi_{T,2}\circ \uppsi_{T,1}$, where $\uppsi_{T,1}$ and $\uppsi_{T,2}$ are given by  \eqref{eq:uppsi1T} and \eqref{eq:uppsi2T}.

First, let us describe the map 
\begin{align*}
\uppsi_{1,T}:\mc{C}_T(\g)\ra 
\on{Com}(\pi_{T+h^{\vee}+1}, H^{\frac{\infty}{2}+0}(L\mf{n}, V_T(\g)\* L_1(\g))).
\end{align*}
As $\on{KL}_T$ is semisimple,
we have
\begin{align*}
V_T(\g)\* L_1(\g)\cong \bigoplus_{\lam\in P_+}\V_{T+1}(\lam)\* \mc{B}_{\lam},
\quad \mc{B}_{\lam}= \on{Hom}_{\affg}(\V_{T+1}(\lam),V_T(\g)\* L_1(\g)).
\end{align*}
It follows that
\begin{align*}
H^{\frac{\infty}{2}+0}(L\mf{n},V_T(\g)\* L_1(\g))\cong \bigoplus_{\lam\in P_+}H^{\frac{\infty}{2}+0}(L\mf{n},\V_{T+1}(\lam))
\* \mc{B}_{\lam}=\bigoplus_{\lam\in P_+}\pi_{T+h^{\vee}+1,\lam}\* \mc{B}_{\lam}.
\end{align*}
Hence,
$\on{Com}(\pi_{T+h^{\vee}+1}, H^{\frac{\infty}{2}+0}(L\mf{n},V_T(\g)\* L_1(\g)))= \mc{B}_0$.
Therefore,
$$\uppsi_{1,T}:\mc{C}_T(\g) \isomap \mc{B}_0 .$$

Next,
by applying the exact functor
$?\* L_1(\g)$
to the resolution of $V_T(\g)$ described in Proposition \ref{Pro:resolution-for-generic-level},
we obtain the resolution
\begin{align}
&0\ra V_T(\g)\* L_1(\g)\ra C_0'\overset{d_0\* 1}{\ra} C_1'\ra\dots \ra C_n'\ra 0,\label{eq:second-resol}\\
& C_i'=\bigoplus_{w\in W\atop \ell(w)=i}\Wak{w\circ 0}{T}\* L_1(\g) \nonumber
\end{align}
of $V_k(\g)\* L_1(\g)$.
As 
each $\Wak{w\circ 0}{T}\* L_1(\g)$ is acyclic with respect to
$H^{\frac{\infty}{2}+i}(L\mf{n} ,?)$,
we can compute the cohomology
$H^{\frac{\infty}{2}+\bullet}(L\mf{n},V_T(\g)\* L_1(\g))$ using the resolution \eqref{eq:second-resol}.
That is,
$H^{\frac{\infty}{2}+\bullet}(L\mf{n},V_T(\g)\* L_1(\g))$
is
the cohomology of the induced complex
\begin{align*}
&0\ra  H^{\frac{\infty}{2}+0}(L\mf{n},C_0')\overset{}{\ra} H^{\frac{\infty}{2}+0}(L\mf{n},C_1')\ra\dots \ra H^{\frac{\infty}{2}+0}(L\mf{n},C_n')\ra 0,\
\end{align*}
of $H^{\frac{\infty}{2}+0}(L\mf{n},\Wak{0}{T}\* L_1(\g))$-modules.
It follows that we have the exact sequence
\begin{align*}
0\ra H^{\frac{\infty}{2}+0}(L\mf{n},V_T(\g)\* L_1(\g))\overset{\uppsi_{2,T}}{\ra}
H^{\frac{\infty}{2}+0}(L\mf{n},\Wak{0}{T}\* \L_1(\nu))\\
\overset{[d_0\* 1]}{\ra} \bigoplus_{i} H^{\frac{\infty}{2}+0}(L\mf{n},\Wak{-\alpha_i}{T}\* \L_1(\nu)),
\end{align*}
where $[d_0\* 1]$ denotes  the map
induced by $d_0\* 1$.
Hence,
$\uppsi_{T}(\mc{C}_T(\g))$ is the kernel of the map $[d_0\* 1]$.
It follows from
  Proposition \ref{Pro:key}
that
\begin{align*}
\uppsi_{T}(\mc{C}_T(\g))=\bigcap_i\on{Ker} (\int S_i^W(z) dz:
\pi_{T'+h^{\vee}}
\ra  \pi_{T'+h^{\vee},-\alpha_i})=\Miura(\W^{T'}(\g))
\end{align*}
as required.
\qed.

\section{Generic decomposition}
Having established Theorem \ref{Th:universalGKO},
 we set
$T=\C$ with the structure map $\tau(K)=k$
for the rest of the paper.
In this section
 we assume further  that $k\not\in \mathbb{Q}$,
so that 
$\on{KL}_k$ is semisimple
and $\V_k(\lam)=\L_k(\lam)$ for all $\mu\in P_+$.
Also,
$\W^{\ell}(\g)=\W_{\ell}(\g)$ by Theorem \ref{Th:+-and-},
where $\ell$ is as in \eqref{eq:ellintro}.

For $\mu\in P_+$ and $\nu\in P_+^1$,
we have
\begin{align*}
\V_k(\mu)\* \L_1(\nu)=\bigoplus_{\lam\in P_+}\V_{k+1}(\lam)\* \mc{B}_{\lam}^{\mu, \nu},
\quad \mc{B}_{\lam}^{\mu, \nu}=
\on{Hom}_{\affg}(\V_{k+1}(\lam),\V_k(\mu)\* \L_1(\nu)).
\end{align*}
According to  Theorem \ref{Th:universalGKO},
each $\mc{B}_{\lam}^{\mu, \nu}$ is a $\W^{\ell}(\g)$-module.
\begin{Th}\label{thm:genericdecomp}
Suppose that $k\not\in \mathbb{Q}$.
For $\lam,\mu\in P_+$ and $\nu\in P_+^1$,
we have
$$\mc{B}_{\lam}^{\mu, \nu}\cong
\begin{cases}  \mathbf{L}_\ell(\chi_{\mu-(\ell+h^{\vee})\lam})&\text{if }\lam-\mu-\nu\in Q,\\
0&\text{otherwise,}
\end{cases}$$
as  $\W^{\ell}(\g)$-modules.
\end{Th}
\begin{proof}
The proof is similar to that 
of Theorem \ref{Th:iso}.
By Corollary \ref{Co:generic-BRST-cohomology},
we have
\begin{align*}
H^{\frac{\infty}{2}+i}(L\mf{n} ,\V_k(\mu)\* \L_1(\nu))
= \bigoplus_{\lam\in P_+}H^{\frac{\infty}{2}+i}(L\mf{n} ,\V_{k+1}(\lam))\* \mc{B}_{\lam}^{\mu, \nu}
\\\cong \bigoplus_{\lam\in P_+}\left(\bigoplus_{w\in W\atop
\ell(w)=i}\pi_{k+h^{\vee}+1,w\circ \lam}\right)\* \mc{B}_{\lam}^{\mu, \nu}.
\end{align*}
Thus,
\begin{align*}
\mc{B}_{\lam}^{\mu, \nu}\cong \Hom_{\pi_{k+h^{\vee}+1}}(\pi_{k+h^{\vee}+1,\lam},H^{\frac{\infty}{2}+0}(L\mf{n} ,\V_k(\mu)\* \L_1(\nu)))
\end{align*}
as modules  over $\W^{\ell}(\g)=\mc{C}_k(\g)$.
On the other hand,
by applying the exact functor
$?\* L_1(\g)$
to the resolution of $\V_k(\mu)=\L_k(\mu)$ described in Proposition \ref{Pro:resolution-for-generic-level-more-general},
we obtain the resolution
\begin{align}
&0\ra \V_k(\mu)\* \L_1(\nu)\ra C_0'\overset{d_0\* 1}{\ra} C_1'\ra\dots \ra C_n'\ra 0,\label{eq:second-resol}\\
& C_i'=\bigoplus_{w\in W\atop \ell(w)=i}\Wak{w\circ \mu}{k}\* \L_1(\nu)\nonumber
\end{align}
of $\V_k(\mu)\* \L_1(\nu)$.
So
we can compute
$H^{\frac{\infty}{2}+\bullet}(L\mf{n},\V_k(\mu)\* \L_1(\nu))$ using the resolution \eqref{eq:second-resol},
that is,
$H^{\frac{\infty}{2}+\bullet}(L\mf{n},\V_k(\mu)\* \L_1(\nu))$
is
the cohomology of the induced complex
\begin{align*}
&0\ra  H^{\frac{\infty}{2}+0}(L\mf{n},C_0')\overset{}{\ra} H^{\frac{\infty}{2}+0}(L\mf{n},C_1')\ra\dots \ra H^{\frac{\infty}{2}+0}(L\mf{n},C_n')\ra 0
\end{align*}
of $H^{\frac{\infty}{2}+0}(L\mf{n},\Wak{0}{k}\* L_1(\g))$-modules.
In particular,
$H^{\frac{\infty}{2}+0}(L\mf{n},\V_k(\mu)\* \L_1(\nu))$
is isomorphic to the kernel of the
map
\begin{align}
 H^{\frac{\infty}{2}+0}(L\mf{n},\Wak{\mu}{k}\* \L_1(\nu))\ra \bigoplus_{i=1}^{\on{rank}\g} H^{\frac{\infty}{2}+0}(L\mf{n},\Wak{\mu-n_i\alpha_i}{k}\* \L_1(\nu))
 \label{eq:the map}
\end{align}
induced by $d_0\* 1$,
where $n_i=\bra \mu+\rho,\alpha_i^{\vee}\ket$.
Therefore from
Propositions \ref{Pro:com-semiinf-Wakimoto} and  \ref{Pro:key}
one finds that for $\lam\in P_+$
\begin{align*}
\mc{B}^{\mu,\nu}_{\lam}=\bigcap_{i=1}^{\on{rank}\g}\on{Ker} (S^W(n_i,\Gamma_i):
\pi^{\ell+h^{\vee}}_{\mu-(\ell+h^{\vee})\lam}
{\longrightarrow} \pi^{\ell+h^{\vee}}_{\mu-(\ell+h^{\vee})\lam-n_i\alpha_i})
\end{align*}
if $\lam-\mu-\nu\in Q$,
and $\mc{B}^{\mu,\nu}_{\lam}=0$ otherwise.
We are done by Proposition \ref{Pro:W-gen-Weyl}.
\end{proof}

\begin{Co}
Suppose that $k\not \in \Q$.
Then $V_{k+1}(\g)$ and $\W^{\ell}(\g)$ form a dual pair in 
$V_k(\g)\* L_1(\g)$.
\end{Co}

\section{Coset construction of minimal series $W$-algebras}
Recall that
the level $k$ is called   an {\em admissible number} for $\affg$
if $L_k(\g)$ is an admissible representation \cite{KacWak89}.
Since we have assumed that $\g$ is
 simply laced,
this condition is equivalent to that
\begin{align}
k+h^{\vee}=\frac{p}{q},\quad p,q\in \Z_{\geq 1},\
(p,q)=1, \ p\geq h^{\vee}.
\label{eq:k-p-q}
\end{align}
An admissible number $k$ is called {\em non-degenerate} (\cite{FKW92}, see also \cite{Ara09b})
if the associated variety (\cite{Ara12}) of $L_k(\g)$ is
exactly the nilpotent cone of $\g$,
which is equivariant to the condition $q\geq h^{\vee}$ for a simply laced $\g$.
The main result of \cite{A2012Dec}
states that 
$\W_k(\g)$ is is rational (and lisse \cite{Ara09b}) if $k$ is an non-degenerate admissible number.
The rational $W$-algebra $\W_{p/q-h^{\vee}}(\g)$ 
for $p,q\in \Z_{\geq 1}$,
$(p,q)=1$,  $p,q \geq h^{\vee}$
is called the {\em $(p,q)$-minimal series principal $W$-algebra} and 
 denoted also by $\W_{p,q}(\g)$.
 Note that $\W_{p,q}(\g)\cong \W_{q,p}(\g)$ 
 by the Feigin-Frenkel duality Theorem \ref{Th:FFduality}.

If $k$ is an admissible number,
then
the level $\ell$,
defined by \eqref{eq:ellintro},
 is given by the formula
\begin{align}
\ell+h^{\vee}=\frac{p}{p+q},
\label{eq:ell-p-q}
\end{align}
so that  $\ell$ is a  non-degenerate admissible number 
and $\W_{\ell}(\g)$ is the $(p+q,q)$-minimal series principal $W$-algebra $\W_{p,p+q}(\g)$.

Let $k$ be an admissible level,
\begin{align*}
Pr_k=\{\lam\in \h^*\mid \hat{\lam}:=\lam+k\Lam_0\text{ is a principal admissible weight of }\affg\}. 
\end{align*}
Here a weight $\hat{\lam}$  of $\affg$ is called principal admissible if 
$\hat{\Delta}(\hat{\lam})\cong \hat{\Delta}_{re}$ as root systems (\cite{KacWak89}).
By \cite{A12-2},
any $L_k(\g)$-module that belongs to $\mc{O}_k$ is completely reducible
and
$\{\L_k(\lam)\mid \lam\in Pr_k\}$ is exactly the set of isomorphism classes  of simple $L_k(\g)$-modules
that belong to in $\mc{O}_k$.

For $k$ an admissible level,
$\mu\in Pr_k$,
$\nu\in P_1^+$.
By  \cite{KacWak90},
we have
\begin{align}
\L_k(\mu)\otimes \L_1(\nu)\cong \bigoplus_{\lam\in Pr^{k+1}
\atop \lam-\mu-\nu\in Q}\L_{k+1}(\lam)\otimes 
\mc{K}_{\mu,\nu}^{\lam},
\label{eq:KW-dec}
\end{align}
where
$$\mc{K}_{\mu,\nu}^{\lam}:=\Hom_{\affg}(\L_{k+1}(\lam),\L_k(\mu)\otimes \L_1(\nu))=
\Hom_{\affg}(\V_{k+1}(\lam),\L_k(\mu)\otimes \L_1(\nu)).$$
Note that
$\mc{K}_{0,0}^0=\on{Com}(L_{k+1}(\g), L_k(\g)\*L_1(\g))=(L_k(\g)\*L_1(\g))^{\g[t]}$.
The
character of each $\mc{K}_{\mu,\nu}^{\lam}$ has been computed in \cite{KacWak90}.
\begin{Th}\label{Th:minimal-sereis}
For an admissible level $k$,
we have
$$\on{Com}(L_{k+1}(\g), L_k(\g)\*L_1(\g))\cong \W_{\ell}(\g)$$ as vertex algebras.
In particular, $\on{Com}(L_{k+1}(\g), L_k(\g)\* L_1(\g))$ is simple,
rational and lisse.
\end{Th}
\begin{proof}
By \cite{CL},
$\on{Com}(L_{k+1}(\g), L_k(\g)\*L_1(\g))$ is a quotient vertex algebra of $C_k(\g)$.
Indeed,
since $V_{k+1}(\g)$ is projective in $\on{KL}_{k+1}$ by Lemma \ref{lem:projective},
the surjection $$V_k(\g)\*L_1(\g)\ra L_k(\g)\*L_1(\g)$$ induces the surjection
\begin{align*}
 C_k(\g)=&\Hom(V_{k+1}(\g), V_k(\g)\*L_1(\g))\ra \Hom(V_{k+1}(\g), L_k(\g)\*L_1(\g)).
\end{align*}
On the other hand,
 \eqref{eq:KW-dec} implies that 
\begin{align*}
\Hom(V_{k+1}(\g), L_k(\g)\*L_1(\g))= 
\Hom(L_{k+1}(\g), L_k(\g)\*L_1(\g))\\=\on{Com}(L_{k+1}(\g), L_k(\g)\*L_1(\g)).
\end{align*}
Thus
$\on{Com}(L_{k+1}(\g), L_k(\g)\*L_1(\g))$ is a quotient 
of $\W^{\ell}(\g)$
by Theorem \ref{Th:universalGKO}.
We are done 
 since 
 it has been already shown in \cite{KacWak90}
 that
 the character of the coset
 $\on{Com}(L_{k+1}(\g), L_k(\g)\*L_1(\g))$ coincides with  that of 
 $\W_{\ell}(\g)$.
\end{proof}

\begin{Rem}
By Theorem \ref{Th:FFduality},
Theorem \ref{Th:minimal-sereis}
realizes {\em all} the minimal series principal $W$-algebras for $ADE$ types.
\end{Rem}

We continue to assume that 
$k$ and $\ell$ are admissible levels defined by
\eqref{eq:k-p-q}  and \eqref{eq:ell-p-q}.
Recall that 
\begin{align*}
\{\mathbf{L}_{\ell}(\chi_{\mu-(\ell+h^{\vee})(\lam+\rho^{\vee})})\mid \lam\in P_+^{p+q-h^{\vee}},\ \mu\in P_+^{p-h^{\vee}}\}
\end{align*}
gives 
the complete set of isomorphism classes of simple $\W_{\ell}(\g)$-modules (\cite{FKW92,A2012Dec}).
Observe 
that 
\begin{align*}
P_+^{p-h^{\vee}}=Pr_{k}\cap P,\quad P_+^{p+q-h^{\vee}}=Pr_{k+1}\cap P.
\end{align*}
Note that $Pr_{k}\cap P=Pr_k=P_+^k$
if $k\in \Z_{\geq 1}$.
\begin{Th}\label{Th:minimal-series-modules}
Let $\mu\in P_+^{p-h^{\vee}}\subset Pr_k$,
$\nu\in P_+^1$.
We have
\begin{align*}
\L_k(\mu)\otimes \L_1(\nu)\cong \bigoplus_{\lam\in P^{p+q-h^{\vee}}_+\atop \lam-\mu-\nu\in Q}
\L_{k+1}(\lam)\otimes  \mathbf{L}_{\ell}(\chi_{\mu-(\ell+h^{\vee})\lam})
\end{align*}
as $L_{k+1}(\g)\* \W_{\ell}(\g)$-modules.
\end{Th}
Note that
since  the coset
$P/Q$ can be  represented by nonzero elements of $P_+^1$,
any simple $\W_{\ell}(\g)$-module appears 
in the decomposition of $\L_k(\mu)\otimes \L_1(\nu)$ for some $\mu\in P_+^{p-h^{\vee}}$ and
$\nu\in P_+^1$.

\begin{proof}[Proof of Theorem \ref{Th:minimal-series-modules}]
By Theorem 6.13 of \cite{A-BGG},
we have 
\begin{align*}
H^{\frac{\infty}{2}+i}(L\mf{n},\L_{k+1}(\lam))=\bigoplus_{w\in \widehat{W}(\lam)\atop 
\ell^{\frac{\infty}{2}}(w)=i}\pi_{k+h^{\vee}+1,w\circ_{k+1} \lam}.
\end{align*}
for $\lam\in Pr_{k+1}\cap P$,
where 
$\ell^{\frac{\infty}{2}}(w)$ is the semi-infinite length of $w\in \widehat{W}(\hat \mu)$
and
\begin{align*}
w\circ_k \lam=w\circ \lam\quad (w\in W),\quad 
t_{\mu}\circ_k \lam=\lam+(k+h^{\vee})\mu\quad (\mu\in Q^{\vee}).
\end{align*}
By \eqref{eq:KW-dec}, it follows that
\begin{align}
H^{\frac{\infty}{2}+i}(L\mf{n},\L_k(\mu)\* \L_1(\nu))\cong \bigoplus_{\lam\in P_+^{p+q-h^{\vee}}}
\bigoplus_{w\in \widehat{W}(\lam)\atop 
\ell^{\frac{\infty}{2}}(w)=i}\pi_{k+h^{\vee}+1,w\circ_{k+1} \lam}\* \mc{K}^{\lam}_{\mu,\nu}.
\label{eq:compare-this}
\end{align}
Thus,
\begin{align}
\mc{K}^{\lam}_{\mu,\nu}\cong \Hom_{\pi_{k+h^{\vee}+1}}(\pi_{k+h^{\vee}+1, \lam},
H^{\frac{\infty}{2}+0}(L\mf{n},\L_k(\mu)\* \L_1(\nu)))
\label{eq:compare-this2}
\end{align}
as 
$\W^{\ell}(\g)$-modules.

On the other hand,
by \cite{A-BGG},
there exists a complex 
\begin{align*}
C^{\bullet}:\dots \ra C^{-1}\ra C^0\overset{d_0}{\ra} C^1\ra\dots \ra C^n\ra \dots
\end{align*}
in $\mc{O}_k$
such that 
\begin{align*}
C^{i}=\bigoplus_{w\in \widehat{W}(\hat \mu)\atop 
\ell^{\frac{\infty}{2}}(w)=i}\Wak{w\circ_k \mu}{k},
\quad \text{and}\quad H^i(C^{\bullet})\cong \delta_{i,0}\L_k(\mu).
\end{align*}
Thus,
$C^{\bullet}$ is a two-sided resolution of $\L_k(\mu)$.
So
$C^{\bullet}\* \L_1(\nu)$ is a two-sided resolution of $\L_k(\mu)\* \L_1(\nu)$.

Let $C^{\bullet}_W$ denote the complex $H^{\frac{\infty}{2}+0}(L\mf{n},C^{\bullet}\* \L_1(\nu))$.
Since each $\Wak{w\circ \mu}{k}\* \L_1(\nu)$ is acyclic with respect to $H^{\frac{\infty}{2}+0}(L\mf{n},?)$,
we have
\begin{align}
H^{\frac{\infty}{2}+i}(L\mf{n},\L_k(\nu)\* \L_1(\nu))\cong H^{i}(C^{\bullet}_W).
\label{eq:iso-as-complex}
\end{align}
Note that $C^{\bullet}_W$ is a complex of $H^{\frac{\infty}{2}+i}(L\mf{n},\Wak{0}{k}\* L_1(\g))$-modules,
and hence is a complex of 
$\W^{\ell}(\g)$-modules,
where $\W^{\ell}(\g)$ acts via the embedding $\W^{\ell}(\g)=\mc{C}_k(\g)\hookrightarrow \pi_{\ell+h^{\vee}}
=\on{Com}(\pi_{k+h^{\vee}+1}, H^{\frac{\infty}{2}+i}(L\mf{n},\Wak{0}{k}\* L_1(\g)))$.
Therefore,
\eqref{eq:iso-as-complex}
is an isomorphism of $\W^{\ell}(\g)$-modules.

We have
\begin{align}
C^{i}_W
=\bigoplus_{w\in \widehat{W}(\hat \mu)\atop 
\ell^{\frac{\infty}{2}}(w)=i}\bigoplus_{\lam\in \h^*
\atop \lam-w\circ_k \mu-\nu\in Q}\pi_{k+h^{\vee}+1,\lam}\* \pi_{\ell+h^{\vee},w\circ_k \mu-(\ell+h^{\vee})\lam}
\label{eq:andthis}
\end{align}
 by Proposition \ref{Pro:com-semiinf-Wakimoto}.
 Therefore,
 by \eqref{eq:compare-this2},
we find that $\mc{K}^{\lam}_{\mu,\nu}$ is a subquotient of 
$ \pi_{\ell+h^{\vee},\mu-(\ell+h^{\vee})\lam}$ as a $\W^{\ell}(\g)$-module,
as $\hat\mu$ is regular dominant.
On the other hand,
$\mc{K}_{\mu,\nu}^{\lam}$
is a module over  the rational quotient $\W_{\ell}(\g)$
and
$\mathbf{L}_{\ell}(\chi_{\mu-(\ell+h^{\vee})\lam})$
is the unique simple $\W_{\ell}(\g)$-module that appears in the local composition factor 
of $ \pi_{\ell+h^{\vee},\mu-(\ell+h^{\vee})\lam}$ (\cite{Ara07}).
Since
the character of $\mc{K}_{\mu,\nu}^{\lam}$ coincides with 
that of $ \mathbf{L}_{\ell}(\chi_{\mu-(\ell+h^{\vee})\lam})$
\cite{KacWak90},
we conclude that
$\mc{K}_{\mu,\nu}^{\lam}\cong \mathbf{L}_{\ell}(\chi_{\mu-(\ell+h^{\vee})\rho^\vee})$.
\end{proof}

\begin{Co}
For an admissible level $k$,
Then $L_{k+1}(\g)$ and $\W_{\ell}(\g)$ form a dual pair in 
$L_k(\g)\* L_1(\g)$.
\end{Co}

Let $\g$ be simply laced as before,
and let  $k$ be an admissible level for $\affg$ and let $n\in \mathbb Z_{>0}$. For $i=0, 1, \dots, n-1$
define the rational numbers $\ell_i$ by
the formula 
\begin{align}
\ell_i+h^{\vee}=\frac{k+i+h^{\vee}}{k+i+h^{\vee}+1},
\label{eq:ell}
\end{align}
which are non-degenerate admissible levels for $\affg$
so that each $\W_{\ell_i}(\g)$ is a minimal series $W$-algebra.
Then the invariant subspace
$(L_k(\g)\* L_1(\g)^{\otimes n})^{\g[t]}$ 
with respect to the 
diagonal action of $\g[t]$
is naturally a vertex subalgebra of $L_k(\g)\* L_1(\g)^{\otimes n}$
and by Theorem \ref{Th:minimal-sereis} it is a vertex algebra extension of $\W_{\ell_0}(\g) \* \W_{\ell_1}(\g) \* \dots \* \W_{\ell_{n-1}}(\g)$ and as such rational and lisse as well by \cite[Theorem 3.5]{HKL}. It is also simple by \cite[Lemma 2.1]{ACKL}.
\begin{Co}\label{Co:rationalprod} 
Let $\g$ be simply laced,
$k$ an admissible number,
$n$ a positive integer.
Then
the coset vertex algebra
$(L_k(\g)\* L_1(\g)^{\otimes n})^{\g[t]}$ 
 is simple, rational and lisse. 
 \end{Co}

A vertex operator algebra $V$ is called
{\em strongly unitary} if $V$ is unitary 
and any positive energy representation of $V$ is unitary in the sense of \cite{DongLin}.
\begin{Th}[Classification of unitary minimal series $W$-algebras of type $ADE$]\label{Th:Unitarity}
The $(p,q)$-minimal series principal $W$-algebra $\W_{p,q}(\g)$ is strongly unitary if and only if
$|p-q|=1$.
 \end{Th}
\begin{proof}
Suppose that $|p-q|=1$.
Since $\W_{p,q}(\g)$ is rational, it is sufficient to show any simple $\W_{p,q}(\g)$-module
$M$ is unitary.
By Theorem \ref{Th:FFduality},
we may assume that $p=q-1$.
Then
by Theorem \ref{Th:minimal-series-modules},
\begin{align*}
M\cong \Hom_{\affg}(\L_{k+1}(\lam),\L_{k}(\mu)\* \L_1(\nu))
\end{align*}
for some $\lam\in P^{k-h^{\vee}+1}_+$, $\mu\in P^k_+$,
$\nu\in P^1_+$,
where $k=p-h^{\vee}\in \Z_{\geq 0}$.
Since both $L_{k+1}(\lam)$ and 
 $\L_{k}(\mu)\* \L_1(\nu)$ are unitary,
$M$ is unitary
is as well
 (cf.\ Corollary 2.8 of \cite{{DongLin}}).
 Conversely, suppose that 
 $\W_{p,q}(\g)$ is strongly unitary.
 Since any  $\W_{p,q}(\g)$-module  is a direct sum of unitary modules of the Virasoro algebra,
 the lowest $L_0$-eigenvalue of any simple $\W_{p,q}(\g)$-module is non-negative.
 Therefore,
 the {\em effective central charge}  \cite{DongMason}
 of $\W_{p,q}(\g)$ coincides with the central charge of $\W_{p,q}(\g)$,
 which equals to 
 \begin{align}
-\frac{r( (h^{\vee}+1)p-h^{\vee}q)(h^{\vee}p-(h^{\vee}+1)q)}{pq},
\label{eq:ccpq}
\end{align}
where $r$ is the rank of $\g$.
On the other hand, Corollary 3.3 of  \cite{DongMason}
implies that the effective central charge coincides the {\em growth} 
(\cite[Exercise 13.39]{Kac90})
of the normalized character of $\W_{p,q}(\g)$,
which is given by 
\begin{align}
r-\frac{h^{\vee}}{pq}\dim \g
\label{eq:growth}
\end{align}
in Theorem 2.16 of \cite{KacWak08}.
Using the formula $\dim \g=r(h^{\vee}+1)$ (note that $\g$ is simply laced),
we find that 
 \eqref{eq:ccpq}  coincides with  \eqref{eq:growth} if and only if $|p-q|=1$.
\end{proof}
Note that for $\g=\mf{sl}_2$,
the above series is exactly the discrete series \cite{GodKenOli86} of the Virasoro algebra.

\section{Level-rank Duality}

We derive another coset realization of the rational $W$-algebras of types $A$ and $D$. These are called level-rank duality in the physics of conformal field theories, but they have first appeared in work by Igor Frenkel \cite{IgorFrenkel} on vertex algebras of type $A$. There, the decomposition of $L_1(\mathfrak{gl}_{mn})$ into modules of $L_n(\mathfrak{sl}_m) \* L_m(\mathfrak{gl}_n)$ was studied. This level-rank duality has then be investigated further in the context of conformal field theory \cite{Walton, ABI90, NT92}.
We can rephrase it as follows. By $\mathcal H$ we denote the rank one Heisenberg vertex algebra.
\begin{Th}\label{th:levelrankA}
For positive integers $n, m, \ell \geq 2$, one has
\[
\on{Com}\left(L_{m+\ell}(\mathfrak{sl}_n), L_{m}(\mathfrak{sl}_n)\otimes L_{\ell}(\mathfrak{sl}_n)  \right) \cong 
\on{Com}\left(L_{n}(\mathfrak{sl}_m) \otimes L_{n}(\mathfrak{sl}_\ell) \otimes \mathcal H, L_{n}(\mathfrak{sl}_{m+\ell}) \right).
\]
Especially it follows that
\[
\on{Com}\left(L_{n}(\mathfrak{sl}_m)  \otimes \mathcal H, L_{n}(\mathfrak{sl}_{m+1}) \right) \cong
\W_k(\mathfrak{sl}_n)
\]
at level $k$ satisfying $k+n= (m+n)/(m+n+1)$.
\end{Th}
\begin{proof}
Let $n, m, \ell$ be three fixed positive integers. Let $G_m, G_\ell, G_{m+\ell}$ be the vertex algebras of $nm, n\ell, n(m+\ell)$ fermionic ghosts-systems, so that $G_{m+\ell}\cong G_m\otimes G_\ell$. Then $G_m$ is strongly generated by the fields $\psi_{i, j}$ and $\psi^*_{j, i}$ for integers $i, j$ in $1\leq i\leq n, 1\leq j\leq m$ and operator products
\[
\psi_{i,j}(z) \psi^*_{k, r}(w) \sim \delta_{j, k}\delta_{i, r}(z-w)^{-1}.
\]
$G_m$ has various interesting vertex subalgebras. Its even subalgebra is isomorphic to $L_1(\mathfrak{gl}_{nm})$ and the $n^2$, respectively $m^2$, fields of the form
\[
\sum_{j=1}^{m} : \psi_{i, j}(z) \psi^*_{j, k}(z): \qquad \text{respectively} \qquad \sum_{i=1}^{n} : \psi_{i, j}(z) \psi^*_{r, i}(z):
\]
generate vertex operators algebras isomorphic to $L_{m}(\mathfrak{gl}_n)$ respectively $L_{n}(\mathfrak{gl}_m)$.
Furthermore $L_{m}(\mathfrak{sl}_n)$ and $L_{n}(\mathfrak{gl}_m)$ form a mutually commuting pair in $G_m$ \cite[Thm. 4.1]{Ostrik} (which used \cite{Walton}).
Similarly, one easily realizes $L_{m+\ell}(\mathfrak{gl}_n)$, $L_{n}(\mathfrak{gl}_m)$, $L_{n}(\mathfrak{gl}_\ell)$ and $L_{n}(\mathfrak{gl}_{m+\ell})$ as vertex subalgebras of $G_{m+\ell}$. We then have the following list of mutually commuting pairs:
\begin{enumerate}
\item $G_m$ and $G_\ell$ in $G_{m+\ell}$
\item $L_{m}(\mathfrak{sl}_n)$ and $L_{n}(\mathfrak{gl}_m)$ in $G_m$
\item $L_{\ell}(\mathfrak{sl}_n)$ and $L_{n}(\mathfrak{gl}_\ell)$ in $G_\ell$
\item $L_{m+\ell}(\mathfrak{sl}_n)$ and $L_{n}(\mathfrak{gl}_{m+\ell})$ in $G_{m+\ell}$
\end{enumerate}
We thus have the following level-rank duality of coset vertex algebras
\begin{equation}\nonumber
\begin{split}
\on{Com}\left(L_{m+\ell}(\mathfrak{sl}_n), L_{m}(\mathfrak{sl}_n)\right.&\left.\otimes L_{\ell}(\mathfrak{sl}_n)  \right) \cong \\
&\cong \on{Com}\left(L_{m+\ell}(\mathfrak{sl}_n),  \on{Com}\left(L_{n}(\mathfrak{gl}_m) \otimes L_{n}(\mathfrak{gl}_\ell), G_{m+\ell} \right)\right)\\
&\cong \on{Com}\left(L_{m+\ell}(\mathfrak{sl}_n)\otimes L_{n}(\mathfrak{gl}_m) \otimes L_{n}(\mathfrak{gl}_\ell), G_{m+\ell} \right)\\
&\cong 
\on{Com}\left(L_{n}(\mathfrak{gl}_m) \otimes L_{n}(\mathfrak{gl}_\ell), \on{Com}\left(L_{m+\ell}(\mathfrak{sl}_n), G_{m+\ell}\right) \right)\\
&\cong 
\on{Com}\left(L_{n}(\mathfrak{gl}_m) \otimes L_{n}(\mathfrak{gl}_\ell), L_{n}(\mathfrak{gl}_{m+\ell}) \right)\\
&\cong 
\on{Com}\left(L_{n}(\mathfrak{sl}_m) \otimes L_{n}(\mathfrak{sl}_\ell) \otimes \mathcal H, L_{n}(\mathfrak{sl}_{m+\ell}) \right)
\end{split}
\end{equation}
where here $\mathcal H$ denotes the rank one Heisenberg vertex algebra in $L_{n}(\mathfrak{sl}_{m+\ell})$ commuting with $L_{n}(\mathfrak{sl}_m) \otimes L_{n}(\mathfrak{sl}_\ell)$.
The case of $\ell=1$ together with Theorem \ref{Th:minimal-sereis} then implies the second statement. 
\end{proof}
This Theorem has a nice rationality corollary. Define $k_i$ by 
$k_i +n = (m+n-i)/(m+n+1-i)$
and let $\mathcal H(\ell)$ be the Heisenberg vertex operator algebra of rank $\ell$. 
Then the coset $\on{Com}\left(L_{n}(\mathfrak{sl}_{m-\ell})  \otimes \mathcal H(\ell), L_{n}(\mathfrak{sl}_{m}) \right)$ is a vertex algebra extension of
$$\W_{k_1}(\mathfrak{sl}_n)\* \W_{k_2}(\mathfrak{sl}_n)\* \dots \* \W_{k_{\ell-1}}(\mathfrak{sl}_n)\* \W_{k_\ell}(\mathfrak{sl}_n)$$
and as such rational and lisse as well by \cite[Theorem 3.5]{HKL}. It is also simple by \cite[Lemma 2.1]{ACKL}.
$L_{n}(\mathfrak{sl}_{m})$ contains the rank $m-1$ Heisenberg vertex algebra. The latter extends to the lattice VOA of the lattice $\sqrt{n}A_{m-1}$ and the $\mathcal H(\ell)$ subVOA of $L_{n}(\mathfrak{sl}_{m})$ extends to the lattice VOA $V_{\Lambda_\ell}$ with $\Lambda_\ell$ the orthogonal complement of $\sqrt{n}A_{m-\ell-1}$ in $\sqrt{n}A_{m-1}$. The coset $\on{Com}\left(L_{n}(\mathfrak{sl}_{m-\ell}), L_{n}(\mathfrak{sl}_{m}) \right)$ is thus a vertex algebra extension of $\on{Com}\left(L_{n}(\mathfrak{sl}_{m-\ell})  \otimes \mathcal H(\ell), L_{n}(\mathfrak{sl}_{m}) \right)\* V_{\Lambda_\ell}$ and as such rational, lisse and simple as well. Summarizing:
\begin{Co}\label{Co:rationalA}
Let $n, m, \ell$ be positive integers, such that $2 \leq m-\ell$. Then the cosets $\on{Com}\left(L_{n}(\mathfrak{sl}_{m-\ell})  \otimes \mathcal H(\ell), L_{n}(\mathfrak{sl}_{m}) \right)$ and $\on{Com}\left(L_{n}(\mathfrak{sl}_{m-\ell}), L_{n}(\mathfrak{sl}_{m}) \right)$ are rational, simple and lisse. 
\end{Co}

The case of type $D$ is a little bit more complicated. We need two statements. Recall that by a simple current one means an invertible object in the tensor category and by a simple current extension of a vertex algebra, one means a vertex algebra extension by a (abelian) group of simple currents. For lisse vertex algebras, this group must be finite. 
Let $n, m, \ell$ be positive integers and $n$ in addition even. 

Firstly, \cite[Theorem 7.4]{KFPX} says 
especially that a simple current extension of $L_n(\mathfrak{so}_m)\otimes L_m(\mathfrak{so}_n)$ embeds conformally in $L_1(\mathfrak{so}_{nm})$. This means that $\on{Com}\left(L_n(\mathfrak{so}_m), L_1(\mathfrak{so}_{nm})\right)$ is a simple current extension of $L_m(\mathfrak{so}_n)$ and the extension can be read off from \cite[Table 3]{KFPX}, see also \cite[Lemmas 4.11 and 4.13]{Lametal17}. 
\begin{equation}\label{eq:D1}
\begin{split}
\on{Com}\left(L_n(\mathfrak{so}_m), L_1(\mathfrak{so}_{nm})\right) &\cong \begin{cases} L_m(\mathfrak{so}_n) & m \ \text{odd} \\
L_m(\mathfrak{so}_n)\oplus \L_m(m\omega_1) & m \ \text{even},
\end{cases}\\
\on{Com}\left(L_n(\mathfrak{so}_m), L_1(\mathfrak{so}_{nm})\right)^{A_{m, n}} &\cong L_m(\mathfrak{so}_n), \qquad A_{m, n} :=\begin{cases} \{1\} & m \ \text{odd} \\
\mathbb Z/2\mathbb Z & m \ \text{even},
\end{cases}\\
\on{Com}\left(L_m(\mathfrak{so}_n), L_1(\mathfrak{so}_{nm})\right) &\cong L_n(\mathfrak{so}_m)\oplus \L_n(n\omega_1).
\end{split}
\end{equation}
with $\omega_0, \omega_1, \dots, \omega_{m/2}$ the fundamental weights of $\mathfrak{so}_m$ in the $m$ even case. 
Here $A_{m, n}$ is the group of simple currents of this extension. 
The second statement we need is that $L_1\left(\mathfrak{so}_{n(m+\ell)}\right)$ is a simple current extension of $L_1\left(\mathfrak{so}_{nm}\right) \otimes L_1\left(\mathfrak{so}_{n\ell}\right)$ for $n, m, \ell \in \mathbb Z_{>0}$ and $n$ even \cite[Sec. 6.1]{Ada16}. This means that 
\begin{equation}\label{eq:D2}
 L_1\left(\mathfrak{so}_{nm}\right) \otimes L_1\left(\mathfrak{so}_{n\ell}\right) \cong L_1\left(\mathfrak{so}_{n(m+\ell)}\right)^{\mathbb Z/2\mathbb Z}.
\end{equation}
Level-rank duality for type $D$ is then
\begin{Th}\label{th:levelrankD}
For positive integers $n, m, \ell$ and $n$ even, let $G=A_{m, n}\times A_{\ell, n}\times \mathbb Z/2\mathbb Z$ with $A_{m, n}=\mathbb Z/2\mathbb Z$ for even $m$ and $A_{m, n}$ trivial otherwise.
Then 
\[
\on{Com}\left(L_{m+\ell}(\mathfrak{so}_n), L_{m}(\mathfrak{so}_n)\*L_{\ell}(\mathfrak{so}_n)  \right)\cong 
\on{Com}\left(L_{n}(\mathfrak{so}_m)\*L_{n}(\mathfrak{so}_\ell), L_{n}(\mathfrak{so}_{m+\ell})\+\L_n(n\omega_1) \right)^{G}.
\]
Especially it follows that
\[
\on{Com}\left(L_{n}(\mathfrak{so}_m), L_{n}(\mathfrak{so}_{m+1})\oplus \L_n(n\omega_1) \right)^{A_{m, n}\times \mathbb Z/2\mathbb Z} \cong
\W_k(\mathfrak{so}_n)
\]
at level $k$ satisfying $k+n-2= (m+n-2)/(m+n-1)$. 
\end{Th}
\begin{proof}
Let $C:= \on{Com}\left(L_{m+\ell}(\mathfrak{so}_n), L_{m}(\mathfrak{so}_n)\otimes L_{\ell}(\mathfrak{so}_n)  \right)$.
Using \eqref{eq:D1} twice, we get
\begin{equation}\nonumber
\begin{split}
C &= \on{Com}\left(L_{m+\ell}(\mathfrak{so}_n), L_{m}(\mathfrak{so}_n)\otimes L_{\ell}(\mathfrak{so}_n)  \right) \\
&\cong \on{Com}\left(L_{m+\ell}(\mathfrak{so}_n), \on{Com}\left(L_n(\mathfrak{so}_m), L_1(\mathfrak{so}_{nm})\right)^{A_{m, n}} \otimes \on{Com}\left(L_n(\mathfrak{so}_\ell), L_1(\mathfrak{so}_{n\ell})\right)^{A_{\ell, n}}   \right)\\
&\cong \on{Com}\left(L_{m+\ell}(\mathfrak{so}_n), \on{Com}\left(L_n(\mathfrak{so}_m), L_1(\mathfrak{so}_{nm})\right) \otimes \on{Com}\left(L_n(\mathfrak{so}_\ell), L_1(\mathfrak{so}_{n\ell})\right)   \right)^{A_{m, n} \times {A_{\ell, n}}} \\
&\cong \on{Com}\left(L_{m+\ell}(\mathfrak{so}_n) \otimes  L_n(\mathfrak{so}_m) \otimes L_n(\mathfrak{so}_\ell), L_1(\mathfrak{so}_{nm})  \otimes  L_1(\mathfrak{so}_{n\ell})   \right)^{A_{m, n} \times {A_{\ell, n}}}. 
\end{split}
\end{equation}
In the third line we used that the actions of $A_{m, n}$ and $L_n(\mathfrak{so}_m)$ commute on $L_1(\mathfrak{so}_{nm})$. Using \eqref{eq:D2} it follows that 
\begin{equation}\nonumber
\begin{split}
C &\cong \on{Com}\left(L_{m+\ell}(\mathfrak{so}_n) \otimes  L_n(\mathfrak{so}_m) \otimes L_n(\mathfrak{so}_\ell), L_1(\mathfrak{so}_{n(m+\ell)})    \right)^{A_{m, n} \times {A_{\ell, n}}\times \mathbb Z/2\mathbb Z}. 
\end{split}
\end{equation}
The claim follows, since $\on{Com}\left(L_{m+\ell}(\mathfrak{so}_n), L_1(\mathfrak{so}_{n(m+\ell)})    \right) \cong L_n(\mathfrak{so}_{m+\ell})\oplus \L_n(n\omega_1)$ by \eqref{eq:D1}.
\end{proof}
We have shown that for even $n$, the coset $\on{Com}\left(L_{n}(\mathfrak{so}_m), L_{n}(\mathfrak{so}_{m+1})\oplus \L_n(n\omega_1) \right)$ has the regular vertex algebra $\W_k(\mathfrak{so}_n)$ as orbifold vertex subalgebra. This in turn means that $\on{Com}\left(L_{n}(\mathfrak{so}_m), L_{n}(\mathfrak{so}_{m+1})\oplus \L_n(n\omega_1) \right)$ is a vertex algebra extension of $\W_k(\mathfrak{so}_n)$. It is regular, since a vertex algebra extension of a regular vertex algebra is regular \cite[Theorem 3.5]{HKL}.

The module $ \L_n(n\omega_1)$ is a simple current for $L_{n}(\mathfrak{so}_{m+1})$ so that  $L_{n}(\mathfrak{so}_{m+1})\oplus \L_n(n\omega_1)$ is a simple current extension of  $L_{n}(\mathfrak{so}_{m+1})$. Conversely, $L_{n}(\mathfrak{so}_{m+1})$ is a $\mathbb Z/2\mathbb Z$-orbifold of $L_{n}(\mathfrak{so}_{m+1})\oplus \L_n(n\omega_1)$. Here $\mathbb Z/2\mathbb Z$ leaves $L_{n}(\mathfrak{so}_{m+1})$ invariant and acts as $-1$ on $\L_n(n\omega_1)$, hence it especially commutes with the action of $L_{n}(\mathfrak{so}_m)$ on $L_{n}(\mathfrak{so}_{m+1})\oplus \L_n(n\omega_1)$ and thus
\begin{equation}\nonumber
\begin{split}
\on{Com}\left(L_{n}(\mathfrak{so}_m), L_{n}(\mathfrak{so}_{m+1}) \right) &= \on{Com}\left(L_{n}(\mathfrak{so}_m), \left(L_{n}(\mathfrak{so}_{m+1})\oplus \L_n(n\omega_1)\right)^{\mathbb Z/2\mathbb Z} \right)\\
&= \on{Com}\left(L_{n}(\mathfrak{so}_m), L_{n}(\mathfrak{so}_{m+1})\oplus \L_n(n\omega_1)\right)^{\mathbb Z/2\mathbb Z}. 
\end{split}
\end{equation}
By \cite[Theorem 5.24]{CarM} the orbifold of a regular vertex operator algebra by a finite cyclic group is regular. Hence $\on{Com}\left(L_{n}(\mathfrak{so}_m), L_{n}(\mathfrak{so}_{m+1}) \right)$ is rational and lisse. It is simple by \cite[Lemma 2.1]{ACKL}. We can even do better by iterating this coset construction. Let $2\leq \ell < m$, then $C_{\ell, m}:=\on{Com}\left(L_{n}(\mathfrak{so}_\ell), L_{n}(\mathfrak{so}_{m}) \right)$ is a vertex algebra extension of   
$C_{\ell, \ell+1} \* C_{\ell+1, \ell+2} \* \dots \* C_{m-2, m-1} \* C_{m-1, m}$
and as such rational, lisse and simple as well.
\begin{Co}\label{Co:rationalD} Let $n, m, \ell$ be integers, such that $n$ is even and $2\leq \ell < m$. Then 
$\on{Com}\left(L_{n}(\mathfrak{so}_\ell), L_{n}(\mathfrak{so}_{m}) \right)$ is simple, rational and lisse. 
\end{Co}

\section{Kazama-Suzuki coset vertex superalgebras}

Let $n, m$ be positive integers. 
In the physics of sigma models for string theories the Kazama-Suzuki \cite{KS} coset
\[
KS(n, m):= \on{Com}\left(L_{n+1}(\mathfrak{sl}_m)\* V_{\sqrt{m(m+1)(m+n+1)}\mathbb Z}, L_n(\mathfrak{sl}_{m+1})\*G_m\right)
\]
is used as a building block for superconformal field theories on Calabi-Yau manifolds. Here $G_m$ is the vertex algebra of $m$ copies of the fermionic ghost-system, 
and by $V_L$ we mean the rational lattice vertex algebra of the positive definite lattice $L$.
From Theorems \ref{Th:minimal-sereis} and \ref{th:levelrankA} we know that $L_n(\mathfrak{sl}_{m+1})\*G_m$ is a vertex superalgebra extension of $\W_\ell(\mathfrak{sl}_m) \otimes \W_k(\mathfrak{sl}_n) \otimes  L_{n+1}(\mathfrak{sl}_m) \otimes \mathcal H(2)$ with $\ell+m=(n+m)/(n+m+1)=k+n$,
where  $\mathcal H(2)$ is a rank two Heisenberg vertex algebra. A short computation reveals that the latter extends to $V_{\sqrt{m(m+1)(m+n+1)}\mathbb Z}\* V_{\sqrt{mn(m+n+1)}\mathbb Z}$ so that we can conclude that $KS(n, m)$ is a vertex superalgebra extension of $\W_\ell(\mathfrak{sl}_m) \otimes \W_k(\mathfrak{sl}_n) \otimes V_{\sqrt{mn(m+n+1)}\mathbb Z}$ and hence rational by \cite{HKL, CKM}.

\begin{Co}\label{Co:supercoset}
Let $n, m$ be positive integers. Then the coset $KS(n, m)$ is a simple, rational, and lisse vertex superalgebra.
\end{Co}

Level-rank duality of type $D$ also gives rise to an interesting coset vertex superalgebra. It is well known that for $m\geq 3$, $L_1(\mathfrak{so}_m)$ is the even subalgebra of the vertex superalgebra $F(m)$ of $m$ free fermions. 
It follows that the coset vertex superalgebra $\on{Com}\left(L_{n+1}(\mathfrak{so}_m),  L_{n}(\mathfrak{so}_{m+1})\* F(m)\right)$
is an extension of its even subalgebra $\on{Com}\left(L_{n+1}(\mathfrak{so}_m),  L_{n}(\mathfrak{so}_{m+1})\* L_{1}(\mathfrak{so}_m)\right)$. 
But by  \ref{Th:minimal-sereis} and \ref{th:levelrankD} the latter is a vertex algebra extension of $\W_k(\mathfrak{so}_m)^{\mathbb Z/2\mathbb Z} \* \W_\ell(\mathfrak{so}_n)$ if both $n$ and $m$ are even. By \cite{CarM}, if $\mathcal{V}$ is a rational and lisse vertex algebra and $G$ is an abelian group of automorphisms of $\mathcal{V}$, the orbifold $\mathcal{V}^G$ is also rational and lisse. It follows that $\W_k(\mathfrak{so}_m)^{\mathbb Z/2\mathbb Z} \* \W_\ell(\mathfrak{so}_n)$ is rational and lisse. Finally, since
$\on{Com}\left(L_{n+1}(\mathfrak{so}_m), L_{n}(\mathfrak{so}_{m+1})\* F(m)\right)$ is a vertex superalgebra extension of $\W_k(\mathfrak{so}_m)^{\mathbb Z/2\mathbb Z} \* \W_\ell(\mathfrak{so}_n)$, it is rational and lisse as well by \cite{HKL, CKM}.
\begin{Co}\label{Co:supercosetD}
Let $n, m$ be positive even integers. Then the coset
\[
\on{Com}\left(L_{n+1}(\mathfrak{so}_m),  L_{n}(\mathfrak{so}_{m+1})\* F(m)\right)
\]
is a simple, rational, and lisse vertex superalgebra.
\end{Co}


\begin{thebibliography}{AKFPP16}
\bibitem[AKFPP16]{Ada16} Drazen Adamovic, Victor G. Kac, Pierluigi Moseneder Frajria, Paolo Papi and Ozren Perse,
\newblock Finite vs infinite decompositions in conformal embeddings, 
\newblock {\em Commun. Math. Phys.} 348 (2016) 445-473.

\bibitem[AFO18]{AgaFreOko}
Mina Aganagic, Edward Frenkel, and Andrei Okounkov.
\newblock Quantum q-langlands correspondence.
\newblock 
{\em Trans. Moscow Math. Soc.} 79 (2018).

\bibitem[AGT10]{AGT}
Luis~F. Alday, Davide Gaiotto, and Yuji Tachikawa.
\newblock Liouville correlation functions from four-dimensional gauge theories.
\newblock {\em Lett. Math. Phys.}, 91(2):167--197, 2010.


\bibitem[ABI90]{ABI90} Daniel Altschuler, Michel Bauer and Claude Itzykson. \newblock  The branching rules of conformal embeddings. \newblock {\em Comm.Math.Phys.} 132 (1990), no. 2, 349-364.



\bibitem[AK11]{AomKit11}
Kazuhiko Aomoto and Michitake Kita.
\newblock {\em Theory of hypergeometric functions}.
\newblock Springer Monographs in Mathematics. Springer-Verlag, Tokyo, 2011.
\newblock With an appendix by Toshitake Kohno, Translated from the Japanese by
  Kenji Iohara.


\bibitem[Ara04]{Ara04}
Tomoyuki Arakawa.
\newblock Vanishing of cohomology associated to quantized {D}rinfeld-{S}okolov
  reduction.
\newblock {\em Int. Math. Res. Not.}, (15):730--767, 2004.

\bibitem[Ara07]{Ara07}
Tomoyuki Arakawa.
\newblock Representation theory of {$W$}-algebras.
\newblock {\em Invent. Math.}, 169(2):219--320, 2007.

\bibitem[Ara12]{Ara12}
Tomoyuki Arakawa.
\newblock A remark on the {$C_2$} cofiniteness condition on vertex algebras.
\newblock {\em Math. Z.}, 270(1-2):559--575, 2012.

\bibitem[Ara14]{A-BGG}
Tomoyuki Arakawa.
\newblock Two-sided {BGG} resolution of admissible representations.
\newblock {\em Represent. Theory}, 18(3):183--222, 2014.

\bibitem[Ara15a]{Ara09b}
Tomoyuki Arakawa.
\newblock Associated varieties of modules over {K}ac-{M}oody algebras and
  {$C_2$}-cofiniteness of {W}-algebras.
\newblock {\em Int. Math. Res. Not.}, 2015:11605--11666, 2015.

\bibitem[Ara15b]{A2012Dec}
Tomoyuki Arakawa.
\newblock Rationality of {W}-algebras: principal nilpotent cases.
\newblock {\em Ann. Math.}, 182(2):565--694, 2015.

\bibitem[Ara16]{A12-2}
Tomoyuki Arakawa.
\newblock Rationality of admissible affine vertex algebras in the category
  {$\mathcal{O}$}.
\newblock {\em Duke Math. J.}, 165(1):67--93, 2016.

\bibitem[Ara17]{Ara16}
Tomoyuki Arakawa.
\newblock Introduction to {W}-algebras and their representation theory.
\newblock n: Callegaro F., Carnovale G., Caselli F., De Concini C., De Sole A. (eds) {\em Perspectives in Lie Theory. }Springer INdAM Series, vol 19. Springer, Cham.


\bibitem[ACKL17]{ACKL} Tomoyuki Arakawa, Thomas Creutzig, Kazuya Kawasetsu and Andrew R. Linshaw
\newblock Orbifolds and cosets of minimal $W$-algebras. \newblock {\em Commun. Math. Phys.} 355, No. 1 (2017), 339-372.


\bibitem[AJ17]{AraJia}
Tomoyuki Arakawa and Cuipo Jiang.
\newblock Coset vertex operator algebras and {W}-algebras.
\newblock {\em Sci. China Math.} (2017). 
 61 (2), 191--206,


\bibitem[ALY19]{ALY17}
Tomoyuki Arakawa, Ching~Hung Lam, and Hiromichi Yamada.
\newblock Parafermion vertex operator algebras and {W}-algebras.
\newblock {\em Trans. Amer. Math. Soc.} 371 (2019), no. 6, 4277--4301


\bibitem[BBSS88]{Bais}
Sander (F.A.) Bais, Peter~Bouwknegt, Michael Surridge and Kareljan Schoutens.
\newblock Coset construction for extended {V}irasoro algebras.
\newblock {\em Nuclear Phys. B}, 304(2):371--391, 1988.

\bibitem[BM13]{BakMil13}
Bojko Bakalov and Todor Milanov.
\newblock {$\mathcal{W}$}-constraints for the total descendant potential of a
  simple singularity.
\newblock {\em Compos. Math.}, 149(5):840--888, 2013.


\bibitem[BFM]{BeiFeiMaz}
Alexander Beilinson, Boris Feigin, and Barry Mazur.
\newblock Introduction to algebraic field theory on curves.
\newblock {\em preprint}.

\bibitem[B89]{B89} Alexander  A. Belavin. \newblock KdV-Type Equations and $W$-Algebras, In Integrable Systems in Quantum Field Theory. \newblock  Adv. Stud. Pure Math. 19, Academic Press, San Diego, 1989, Pages 117-125

\bibitem[BFN16]{BraFinNak16}
Alexander Braverman, Michael Finkelberg, and Hiraku Nakajima.
\newblock Instanton moduli spaces and {$\mathcal{W}$}-algebras.
\newblock {\em Ast\'erisque}, (385):vii+128, 2016.

\bibitem[BPZ84]{BPZ84}
Alexander~A. Belavin, Alexander~M. Polyakov, and Alexander~B. Zamolodchikov.
\newblock Infinite conformal symmetry in two-dimensional quantum field theory.
\newblock {\em Nuclear Phys. B}, 241(2):333--380, 1984.

\bibitem[CM]{CarM} Scott Carnahan and Masahiko Miyamoto. \newblock Regularity of fixed-point vertex operator subalgebras.
 \newblock {\em arXiv:1603.05645}.


\bibitem[CKM]{CKM} Thomas Creutzig, Shashank Kanade and Robert McRae. \newblock Tensor categories for vertex operator superalgebra extensions. \newblock arXiv:1705.05017.

\bibitem[CL19]{CL}
Thomas Creutzig and Andrew~R. Linshaw.
\newblock Cosets of affine vertex algebras inside larger structures.
\newblock {\em Journal of Algebra} 517 (2019) 396-438.

\bibitem[CG17]{CG}
  Thomas Creutzig and Davide Gaiotto. \newblock Vertex Algebras for S-duality. 
\newblock  arXiv:1708.00875.

\bibitem[CGL18]{CGL}
  Thomas Creutzig, Davide Gaiotto and Andrew~R. Linshaw. \newblock S-duality for the large $N=4$ superconformal algebra. 
\newblock  	arXiv:1804.09821.


\bibitem[CHR12]{CHR}
Thomas Creutzig, Yasuaki Hikida and Peter B. Ronne. \newblock Higher spin AdS$_3$ supergravity and its dual CFT. .\newblock  JHEP {\bf 1202} (2012) 109.
  
\bibitem[CHR13]{CHR2}  Thomas Creutzig, Yasuaki Hikida and Peter B. Ronne. \newblock  $N=1$ supersymmetric higher spin holography on AdS$_3$. \newblock   JHEP {\bf 1302} (2013) 019. 
  
\bibitem[D03]{D03} Leonid A. Dickey. \newblock Soliton equations and Hamiltonian systems, Advanced series in mathematical physics. \newblock World scientific, Vol. 26, 2nd Ed., 2003.  
  
\bibitem[DSK06]{De-Kac06}
Alberto De~Sole and Victor~G. Kac.
\newblock Finite vs affine {$W$}-algebras.
\newblock {\em Japan. J. Math.}, 1(1):137--261, 2006.

\bibitem[DSKV13]{De-KacVal13}
Alberto De~Sole, Victor~G. Kac, and Daniele Valeri.
\newblock Classical {$\mathscr{W}$}-algebras and generalized
  {D}rinfeld-{S}okolov bi-{H}amiltonian systems within the theory of {P}oisson
  vertex algebras.
\newblock {\em Comm. Math. Phys.}, 323(2):663--711, 2013.

\bibitem[DL14]{DongLin}
Chongying Dong and Xingjun Lin.
\newblock Unitary vertex operator algebras.
\newblock {\em J. Algebra}, 397:252--277, 2014.

\bibitem[DM04]{DongMason}
 Chongying Dong and Geoffrey Mason.
\newblock Rational vertex operator algebras and the effective central charge.
\newblock {\em Int. Math. Res. Not}. 2004, no. 56, 2989--3008. 


\bibitem[Fei84]{Feu84}
Boris Feigin. 
\newblock Semi-infinite homology of {L}ie, {K}ac-{M}oody and {V}irasoro
  algebras.
\newblock {\em Uspekhi Mat. Nauk}, 39(2(236)):195--196, 1984.


\bibitem[FF90a]{FF90}
Boris Feigin and Edward Frenkel.
\newblock Quantization of the {D}rinfel\cprime d-{S}okolov reduction.
\newblock {\em Phys. Lett. B}, 246(1-2):75--81, 1990.

\bibitem[FF90b]{FeuFre90}
Boris Feigin and Edward Frenkel.
\newblock Affine {K}ac-{M}oody algebras and semi-infinite flag manifolds.
\newblock {\em Comm. Math. Phys.}, 128(1):161--189, 1990.

\bibitem[FF91]{FeiFre91}
Boris Feigin and Edward Frenkel.
\newblock Duality in {$W$}-algebras.
\newblock {\em Internat. Math. Res. Notices}, (6):75--82, 1991.

\bibitem[FF92]{FeiFre92}
Boris Feigin and Edward Frenkel.
\newblock Affine {K}ac-{M}oody algebras at the critical level and {G}el\cprime
  fand-{D}iki\u\i\ algebras.
\newblock In {\em Infinite analysis, Part A, B (Kyoto, 1991)}, volume~16 of
  {\em Adv. Ser. Math. Phys.}, pages 197--215. World Sci. Publ., River Edge,
  NJ, 1992.
  
  \bibitem[FF96]{FeiFre96}
Boris Feigin and Edward Frenkel.
\newblock Integrals of motion and quantum groups.
\newblock In {\em Integrable systems and quantum groups ({M}ontecatini {T}erme,
  1993)}, volume 1620 of {\em Lecture Notes in Math.}, pages 349--418.
  Springer, Berlin, 1996.

\bibitem[FG18]{FG}
  Edward ~Frenkel and Davide~Gaiotto. \newblock Quantum Langlands Dualities of Boundary Conditions, D-modules, and conformal blocks.
\newblock  	{ arXiv:1805.00203 [hep-th]}.

\bibitem[FHL93]{Frenkel:1993aa}
Igor~B. Frenkel, Yi-Zhi Huang, and James Lepowsky.
\newblock On axiomatic approaches to vertex operator algebras and modules.
\newblock {\em Mem. Amer. Math. Soc.}, 104(494):viii+64, 1993.

\bibitem[Fie06]{Fie06}
Peter Fiebig.
\newblock The combinatorics of category {$\mathcal{O}$} over symmetrizable
  {K}ac-{M}oody algebras.
\newblock {\em Transform. Groups}, 11(1):29--49, 2006.

\bibitem[FJMM16]{FeiJimMiw16}
Boris Feigin, Michio Jimbo, Tetsuji Miwa, and Eugene Mukhin.
\newblock Branching rules for quantum toroidal {$\germ{gl}_n$}.
\newblock {\em Adv. Math.}, 300:229--274, 2016.

\bibitem[FBZ04]{FreBen04}
Edward Frenkel and David Ben-Zvi.
\newblock {\em Vertex algebras and algebraic curves}, volume~88 of {\em
  Mathematical Surveys and Monographs}.
\newblock American Mathematical Society, Providence, RI, second edition, 2004.

\bibitem[FK81]{FreKac80}
Igor~B. Frenkel and Victor~G. Kac.
\newblock Basic representations of affine {L}ie algebras and dual resonance
  models.
\newblock {\em Invent. Math.}, 62(1):23--66, 1980/81.

\bibitem[FKW92]{FKW92}
Edward Frenkel, Victor G. Kac, and Minoru Wakimoto.
\newblock Characters and fusion rules for {$W$}-algebras via quantized
  {D}rinfel\cprime d-{S}okolov reduction.
\newblock {\em Comm. Math. Phys.}, 147(2):295--328, 1992.

\bibitem[FL88]{FatLyk88}
Vladimir~A. Fateev and Sergei~L. Lykyanov.
\newblock The models of two-dimensional conformal quantum field theory with
  {$Z\sb n$} symmetry.
\newblock {\em Internat. J. Modern Phys. A}, 3(2):507--520, 1988.

\bibitem[Fre82]{IgorFrenkel}
Igor~B. Frenkel.
\newblock Representations of affine {L}ie algebras, {H}ecke modular forms and
  {K}orteweg-de {V}ries type equations.
\newblock In {\em Lie algebras and related topics ({N}ew {B}runswick, {N}.{J}.,
  1981)}, volume 933 of {\em Lecture Notes in Math.}, pages 71--110. Springer,
  Berlin-New York, 1982.

\bibitem[Fre92]{Fre92Car}
Edward Frenkel.
\newblock {$\mathscr{W}$}-algebras and {L}anglands-{D}rinfel\cprime d
  correspondence.
\newblock In {\em New symmetry principles in quantum field theory ({C}arg\`ese,
  1991)}, volume 295 of {\em NATO Adv. Sci. Inst. Ser. B Phys.}, pages
  433--447. Plenum, New York, 1992.

\bibitem[Fre05]{Fre05}
Edward Frenkel.
\newblock Wakimoto modules, opers and the center at the critical level.
\newblock {\em Adv. Math.}, 195(2):297--404, 2005.

\bibitem[Fre07]{Fre07}
Edward Frenkel.
\newblock {\em Langlands correspondence for loop groups}, volume 103 of {\em
  Cambridge Studies in Advanced Mathematics}.
\newblock Cambridge University Press, Cambridge, 2007.

\bibitem[FZ92]{FreZhu92}
Igor~B. Frenkel and Yongchang Zhu.
\newblock Vertex operator algebras associated to representations of affine and
  {V}irasoro algebras.
\newblock {\em Duke Math. J.}, 66(1):123--168, 1992.

\bibitem[Gai16]{Gai16}
Dennis Gaitsgory.
\newblock Quantum Langlands Correspondence.
arXiv:1601.05279 [math.AG].

\bibitem[Gai18]{Gai18}
Dennis Gaitsgory.
\newblock
The master chiral algebras.
talk at Perimeter Institute,
https://www.perimeterinstitute.ca/videos/master-chiral-algebra.


\bibitem[Gep88]{Gep}
  David ~Gepner. \newblock Space-Time Supersymmetry in Compactified String Theory and Superconformal Models. \newblock  Nucl.\ Phys.\ B {\bf 296} (1988) 757.

\bibitem[Gen17]{Genra}
Naoki Genra.
\newblock Screening operators for {W}-algebras.
\newblock {\em Sel. Math. New Ser.}, Sel. Math. New Ser. (2017) 23(3): 2157--2202.

\bibitem[GG11]{GG}
Matthias~R. Gaberdiel and Rajesh Gopakumar.
\newblock An ${A}d{S}_3$ dual for minimal model {CFT}s,.
\newblock {\em Phys.\ Rev.\ D}, 83(066007), 2011.

\bibitem[GKO86]{GodKenOli86}
Peter~Goddard, Adrian~Kent, and David~Olive.
\newblock Unitary representations of the {V}irasoro and super-{V}irasoro
  algebras.
\newblock {\em Comm. Math. Phys.}, 103(1):105--119, 1986.

\bibitem[HKL15]{HKL} Yi-Zhi  Huang, Alexander Kirillov, Jr. and James Lepowsky. \newblock  Braided tensor categories and extensions of vertex operator algebras. \newblock  Comm. Math. Phys. 337 (2015), 1143-1159.

\bibitem[JL]{Lametal17} Cuipo Jiang and Ching Hung Lam. \newblock Level-Rank Duality for Vertex Operator Algebras of types $B$ and $D$. \newblock 	arXiv:1703.04889.

\bibitem[Kac90]{Kac90}
Victor~G. Kac.
\newblock {\em Infinite-dimensional {L}ie algebras}.
\newblock Cambridge University Press, Cambridge, third edition, 1990.

\bibitem[KFPX12]{KFPX} 
Victor G. Kac, Pierluigi Moseneder Frajria, Paolo Papi and Feng Xu.
\newblock{Conformal embeddings and simple current extensions} 
\newblock {\em IMRN} (2015), 14, 5229-5288.

\bibitem[KK79]{KacKaz79}
Victor~G. Kac and David~A. Kazhdan.
\newblock Structure of representations with highest weight of
  infinite-dimensional {L}ie algebras.
\newblock {\em Adv. in Math.}, 34(1):97--108, 1979.

\bibitem[KR87]{KR87} Victor G. Kac and Ashok Raina. \newblock Bombay lectures on highest weight representations of infinite dimensional Lie algebras. \newblock World Scientific, Singapore (1987).

\bibitem[KRW03]{KacRoaWak03}
Victor G. Kac, Shi-Shyr Roan, and Minoru Wakimoto.
\newblock Quantum reduction for affine superalgebras.
\newblock {\em Comm. Math. Phys.}, 241(2-3):307--342, 2003.



\bibitem[KW89]{KacWak89}
Victor G. Kac and Minoru Wakimoto.
\newblock Classification of modular invariant representations of affine
  algebras.
\newblock In {\em Infinite-dimensional Lie algebras and groups
  (Luminy-Marseille, 1988)}, volume~7 of {\em Adv. Ser. Math. Phys.}, pages
  138--177. World Sci. Publ., Teaneck, NJ, 1989.

\bibitem[KW90]{KacWak90}
Victor~G. Kac and Minoru Wakimoto.
\newblock Branching functions for winding subalgebras and tensor products.
\newblock {\em Acta Appl. Math.}, 21(1-2):3--39, 1990.

\bibitem[KW08]{KacWak08}
Victor~G. Kac and Minoru Wakimoto.
\newblock On rationality of {$W$}-algebras
\newblock  {\em Transform. Groups} 13 (2008), no. 3-4, 671--713. 

\bibitem[KS89]{KS}
Yoichi Kazama and Hisao Suzuki \newblock New N=2 Superconformal Field Theories and Superstring Compactification. \newblock Nucl.\ Phys.\ B {\bf 321} (1989) 232.



\bibitem[Li05]{Li05}
Haisheng Li.
\newblock Abelianizing vertex algebras.
\newblock {\em Comm. Math. Phys.}, 259(2):391--411, 2005.



\bibitem[MNT10]{MatNagTsu05}
Atsushi Matsuo, Kiyokazu Nagatomo, and Akihiro Tsuchiya.
\newblock Quasi-finite algebras graded by {H}amiltonian and vertex operator
  algebras.
\newblock In {\em Moonshine: the first quarter century and beyond}, volume 372
  of {\em London Math. Soc. Lecture Note Ser.}, pages 282--329. Cambridge Univ.
  Press, Cambridge, 2010.

\bibitem[NT92]{NT92} Tomoki Nakanishi and Akihiro Tsuchiya. \newblock Level-rank duality of WZW models in conformal
field theory. \newblock Comm. Math. Phys. 144 (1992), no. 2, 351-372.

\bibitem[OS14]{Ostrik} Victor Ostrik and Michael Sun. \newblock Level-rank duality via tensor categories.  \newblock  Comm. Math. Phys. 326 (2014) 49-61.

\bibitem[SV13]{SchVas13}
Olivier Schiffmann and Eric Vasserot.
\newblock Cherednik algebras, {W}-algebras and the equivariant cohomology of
  the moduli space of instantons on {$\bold{A}^2$}.
\newblock {\em Publ. Math. Inst. Hautes \'Etudes Sci.}, 118:213--342, 2013.

\bibitem[TK86]{TsuKan86}
Akihiro Tsuchiya and Yukihiro Kanie.
\newblock Fock space representations of the {V}irasoro algebra. {I}ntertwining
  operators.
\newblock {\em Publ. Res. Inst. Math. Sci.}, 22(2):259--327, 1986.

\bibitem[Vor93]{Vor93}
Alexander~A. Voronov.
\newblock Semi-infinite homological algebra.
\newblock {\em Invent. Math.}, 113(1):103--146, 1993.

\bibitem[Vor99]{Vor99}
Alexander~A. Voronov.
\newblock Semi-infinite induction and {W}akimoto modules.
\newblock {\em Amer. J. Math.}, 121(5):1079--1094, 1999.

\bibitem[Wal89]{Walton}Mark A. Walton. \newblock Conformal branching rules and modular invariants. \newblock Nucl. Phys. B 322 (1989), 775-790.

\bibitem[Wan93]{Wan93}
Weiqiang Wang.
\newblock Rationality of {V}irasoro vertex operator algebras.
\newblock {\em Internat. Math. Res. Notices}, (7):197--211, 1993.

\bibitem[Zhu96]{Zhu96}
Yongchang Zhu.
\newblock Modular invariance of characters of vertex operator algebras.
\newblock {\em J. Amer. Math. Soc.}, 9(1):237--302, 1996.



\end{thebibliography}
\end{document}